\documentclass[sort&compress,reqno]{elsarticle}
\usepackage{doi}
\usepackage{url}

\usepackage{graphicx, color,psfrag}
\usepackage{hyperref}
\hypersetup{ 
    colorlinks=true,       
}
\usepackage{amsmath,amssymb,amsthm}
\usepackage{mathtools}
\usepackage[mathscr]{eucal}

\usepackage{paralist}

\newtheorem{thm}{Theorem} 
\newtheorem{lem}[thm]{Lemma} 
\newtheorem{prop}[thm]{Proposition} 

\theoremstyle{remark}

\theoremstyle{definition}
\newtheorem*{example}{Example} 

\newcommand{\field}[1]{ \ensuremath{\mathbb{#1}}}
\newcommand{\T}{\ensuremath{\field{T}}}
\newcommand{\C}{ \ensuremath{\field{C}}}
\newcommand{\R}{ \ensuremath{\field{R}}}

\newcommand{\N}{ \ensuremath{\field{N}}}

\newcommand{\D}{ \ensuremath{\field{D}}}

\newcommand{\disk}{ \ensuremath{\mathbf D}}

\newcommand{\radon}{ \ensuremath{\mathcal R}}
\newcommand{\hilb}{ \ensuremath{\mathcal H}}
\newcommand{\fant}{ \ensuremath{\Phi}}

\newcommand{\cauchy}{ \ensuremath{\mathcal C} }

\def\Xint#1{\mathchoice
{\XXint\displaystyle\textstyle{#1}}
{\XXint\textstyle\scriptstyle{#1}}
{\XXint\scriptstyle\scriptscriptstyle{#1}}%
{\XXint\scriptscriptstyle\scriptscriptstyle{#1}}%
\!\int}
\def\XXint#1#2#3{{\setbox0=\hbox{$#1{#2#3}{\int}$}
\vcenter{\hbox{$#2#3$}}\kern-.5\wd0}}

\def\dashint{\Xint-}

\newcommand{\norm}[1]{ \ensuremath{\left\lVert{#1}\right\rVert}}
\newcommand{\abs}[1]{ \ensuremath{\left\lvert{#1}\right\rvert}}

\DeclareMathOperator{\inter}{int}

\DeclareMathOperator*{\ilim}{i-lim}
\DeclareMathOperator*{\elim}{e-lim}
\begin{document}

\title{Conditioning moments of singular measures for entropy optimization. I\footnote{Israel Gohberg, in memoriam}}
\author{Marko Budi\v{s}i\'{c}}
\address{Department of Mechanical Engineering, University of California, Santa Barbara, CA 93106,\\
mbudisic@engr.ucsb.edu}
\author{Mihai Putinar}
\address{Department of Mathematics, University of California, Santa Barbara, CA 93106,\\ 
mputinar@math.ucsb.edu}
\begin{abstract} In order to process a potential moment sequence by the entropy optimization method one has to be assured that the original measure is absolutely continuous with respect to Lebesgue measure. We propose a non-linear exponential transform of the moment sequence of any measure, including singular ones, so that the entropy optimization method can still be used in the reconstruction or approximation of the original. The Cauchy transform in one variable, used for this very purpose in a classical context by A.\ A.\ Markov and followers, is replaced in higher dimensions by the Fantappi\`{e} transform. Several algorithms for reconstruction from moments are sketched, while we intend to provide the numerical experiments and computational aspects in a subsequent article. The essentials of complex analysis, harmonic analysis, and entropy optimization are recalled in some detail, with the goal of making the main results more accessible to non-expert readers.
\\ \\
Keywords: Fantappi\`e transform; entropy optimization; moment problem; tube domain; exponential transform
\end{abstract}

\maketitle

\section{Introduction}\label{sec:intro}

In sciences and engineering, a particular inverse problem arises often, requiring approximation of a measure by a density function from knowledge of linear data, e.g., integrals of a function basis against a measure. Classical moment problem considers integrals of monomials, the power moments, as the set of known measurements, while the generalized moment problems extend the admissible inputs to integrals of orthogonal polynomials, Fourier basis, wavelets, or other functional bases.

The list of applications of the moment problem is long, ranging from engineering, through physics, statistics, well into applied mathematics. While pure mathematical settings allow for infinite moment sequences, leading to classical moment problems of Hausdorff, Hamburger, and Stieltjes, the applied settings almost exclusively assume knowledge of only a finite number of moments, which is known as 
the truncated moment problem. 

In early 1980s, statistical physics and signal processing communities recognized that a practical solution to the truncated moment problem, which is mathematically under-determined, can be found through optimization of the Shannon entropy, a nonlinear functional acting on the density of the measure \cite{mead1984,jayne1982}. Initial success, in the numerically unfavorable setting of power moments, generated sufficient interest to improve on the original method \cite{junk2000,bandy2005,Abramov:2007to,biswa2010} and arrive at a routinely-used method not only in physics, but also in statistics and control theory \cite{georg2006, borwe2012, hauck2008}. Furthermore, optimization of entropy has been shown to be of theoretical importance: it can be used to fully characterize the moment sequences representable by densities based on truncated moment data \cite{blekh2011}, and arbitrarily incomplete moment data \cite{ambro2011}.

However, not every moment sequence is a suitable input for the entropy optimization. In particular, it is easy to demonstrate that the moment sequence of the Dirac-$\delta$ distribution is not a feasible input, as the optimization does not converge in that case. Such singular measures captured our focus, as we were motivated by potential applications to inverse problems in dynamical systems.

Measures invariant under evolution of dynamical systems are of particular interest, with increasing activity driven by applied problems. On chaotic attractors, trajectories of dynamical systems are known to be non-robust to any errors and behavior is more reliably represented using statistical methods \cite{Eckmann:1985ik}. Surprisingly, even in chaotic regimes, the moment data of invariant measures can be reliably computed from simulated and experimental trajectories by averaging moment functions along them, despite the errors inherent to those procedures \cite{Sigurgeirsson:2001wi}. Singular invariant measures abound in dynamical systems, e.g., a system with an attracting fixed point preserves a Dirac-$\delta$ distribution, whose moments are easily computed by averaging along any trajectory in the basin of attraction. As mentioned before, entropy optimization would not converge for such a common invariant measure.

To overcome the obstacle of singular measures, we propose a three step process:
\begin{inparaenum}[\upshape(\itshape i\upshape)]
\item regularization,
\item entropy optimization, and
\item inversion.
\end{inparaenum} Regularization conditions the moment sequence into a feasible input to the entropy optimization, converting the original moment sequence into moments of a bounded, integrable \emph{phase function}. The entropy optimization step can then be used to recover a closed expression for the phase function approximation. In the inversion step, point-wise evaluations of the phase function are used to recover an approximant of the original measure.

The proposed regularization of the moment sequence of a singular,
positive measure derives from an original idea of A.\ A.\ Markov to study the moment sequence through its complex generating function. We start by a simple observation that an analytic function mapping
a domain into the open upper-half plane admits an analytic logarithm
whose imaginary part (the phase) is bounded from below by $0$ and from
above by $\pi$. The passage from a positive measure to the phase
function through a canonical integral transform, obeying the above
principle, has circulated in the Russian literature in connection with
the century old works devoted to the one dimensional L-problem of
moments. The early articles by M.\ G.\ Krein, N.\ Akhiezer and A.\
Nudelman on the subject offer a comprehensive account of this method
\cite{krein1959,krein1977}.
  
In the present article we go beyond one dimension, considering
Fantappi\`e transforms of positive measures supported by a wedge in
$\R^d$ \cite{ander2004,mccarthy2005}.  The existing methods of
harmonic analysis on tube domains enter naturally into the picture
offering to the maximum entropy reconstruction method a solid
background. The much nicer sequence of moments of the phase function
are obtained from the moment sequence of the original measure via a
non-linear recurrent operation. A thorough investigation of the
multivariate moment via asymptotic expansions of the Fantappi\`{e}
transform of the underlying measure was undertaken by Henkin and
Shananin \cite{Henkin1990,henki1992}, whose work we take as a basis
for ours.

While the entropy optimization provides a standard reconstruction
procedure for the phase function, the approaches to inversion for one- and
multi-variate problems are different. In one-dimensional case, we can
make use of the well known Plemelj-Sokhotski formulas
\cite{henrici1977vol3,king2009v1} to complete the inversion step. The
formulas, however, are difficult to generalize to multivariate
settings \cite{fuks1963}; instead, we propose a ray beam disintegration,
based on a refined and partially forgotten one-dimensional analysis of
the phase regularization due to Aronszajn and Donoghue
\cite{arons1956}. The ray beam approach reduces the problem to a
setting similar to medical tomography, based on inverse Radon or
Laplace transform methods \cite{natte2001,palam2004}. We
believe this will be a fruitful approach that we plan on exploring in
follow-up papers, so we only draft it in this paper.

The paper is organized using the following outline. Section
\ref{sec:preliminaries} briefly introduces the multivariate moment
problem and the entropy optimization, including an example
illustrating lack of convergence for a Dirac-$\delta$ measure. In
Section \ref{sec:maxent1d} we expose the elementary aspects of the
entropy optimization method, in the case of one real variable for
unbounded and bounded supports, using, respectively, power moments and
trigonometric moments, i.e., Fourier coefficients. Section
\ref{sec:phasereg} is devoted to generalization to multivariate
problems, through the phase regularization of the Fantappi\`{e}
transform of a measure supported by a wedge in Euclidean space
(Section \ref{sec:tube}) and by special compact domains in Euclidean
space (Section \ref{sec:restrtube}). A Riesz-Herglotz formula is
derived, in the spirit of \cite{koranyi1963,aizenberg1976}, with a
couple of examples on product domains.

The present article remains at a theoretical level, leaving for a
continuation of it to deal with further practical aspects: numerical experiments, the error analysis and examples from dynamical systems. We do, however, present practical algorithms that
are essential for moment conditioning, the Miller-Nakos algorithm in
\ref{sec:miller}, described in \cite{nakos1993}, and a recent
algorithm for entropy optimization, described in \cite{bandy2005}, in
\ref{sec:fime}.
    
We dedicate this work to the late Israel I.\ Gohberg, legendary figure of modern operator theory and function theory.
His original and highly influential ideas have permanently shaped moment problems and the entropy method referred to in the following pages.

\section{Preliminaries}\label{sec:preliminaries}
Let $d \geq 1$ be a fixed dimension and
let $K$ be a closed subset of the Euclidean space ${\R}^d$.
Fix a finite set $A \subset {\N}^d$ of multi-indices. The
{\it truncated moment problem} with supports on $K$ and monomials
labeled by $A$ consists in finding (as effectively as possible) a
positive measure $\mu$ supported by $K$, with prescribed moments
\begin{equation}
\label{moment problem}
 \gamma_\alpha = \int_K x^\alpha d\mu(x), \ \
\alpha \in A.
\end{equation}
In case the set $K$ is unbounded, it is implicit that the above
integrals converge in Lebesgue sense.
 Throughout this article we adopt the multi-index
notation
$$ x^\alpha = x_1^{\alpha_1} x_2^{\alpha_2}...x_d^{\alpha_d},\ \ \
x \in \R^d.$$

A few basic questions are in order:

\begin{enumerate} \setcounter{enumi}{0}
\item \emph{Characterize all sequences of moments $(a_\alpha)_{\alpha \in
A}$ associated to positive measures carried by the set $K$.}
\end{enumerate}

This question can be rephrased in terms of the formal integration
functional
$$ L(f) = \sum_{\alpha \in A} c_\alpha \gamma_\alpha, \ \  f =
\sum_{\alpha \in A} c_\alpha x^\alpha.$$ Let us denote by
${\R}[x]_A$ the linear span, in the ring of polynomials
${\R}[x]$, of all monomials $x^\alpha, \ \ \alpha \in A$.

A necessary and sufficient condition that a linear functional $L :
{\R}[x] \longrightarrow \R$ is representable by a
positive measure supported by the set $K$ is that $L$ in
non-negative on all elements $f \in {\R}[x]$ which are
non-negative on $K$ . Then $L$ can be extended via a Hahn-Banach
construction to a positive linear functional on the space of
continuous functions on $K$, with polynomial growth at infinity.
This observation remains however of a limited theoretical
importance, and it becomes effective only when simple
characterizations of non-negative polynomials on $K$ is available.
Fortunately, in the case when $K$ is a basic semi-algebraic set,
such "Positivstellens\"atze" were recently resurrected and a good
collection of examples is available, see  \cite{prest2001}.

The single variable case is the simplest and best understood. The
following result goes back to Marcel Riesz \cite{riesz1923}.

\begin{thm} Let $n$ be a fixed degree and $(a,b)$ an
interval on the real line, bounded or not. A positive measure
$\mu$ carried by the closure of $(a,b)$ exists, with moments
$$ \gamma_k = \int x^k d\mu(x), \ \ 0 \leq k <n,$$
and
$$ \gamma_n \geq \int x^n d\mu$$
if and only if the associated functional $L$ satisfies $L(f) \geq
0$ for all polynomials $f(x) = c_0 + c_1 x+...+c_n x^n$ which are
non-negative on $(a,b).$
\end{thm}

Three cases are distinguished, and they correspond to classical
moment problem studies: $(a,b) = (0,1)$, known as the Hausdorff
moment problem, $(a,b) = (0,\infty)$ known as Stieltjes moment
problem, and $(a,b) = (-\infty, \infty)$ known as the Hamburger
moment problem. In each separate situation a full characterization
of all non-negative polynomials on $(a,b)$ is available, with the
result of making the above M. Riesz result effective. We refer the
reader to \cite{akhie1965,shoha1943} for full details.

\begin{enumerate}\setcounter{enumi}{1}
\item \emph{Knowing that problem (\ref{moment problem}) is solvable,
find constructively one particular solution.}
\end{enumerate}

As a general rule, any attempt to solve the truncated problem
(\ref{moment problem}) starts with the observation that the set of
all solutions
$$ \Sigma = \{ \mu \geq 0;\ \ \int_K x^\alpha d\mu = \gamma_\alpha, \ \
\alpha \in A\}$$ is convex and closed in the weak-* topology. If
we include $\alpha =0$ among the elements of the index set $A$,
then all elements of $\Sigma$ have fixed total variation. Thus, in
this case,  on a compact support $K$, the set of solutions
$\Sigma$ is compact in the weak-* topology of all measures.

Among all elements of the solution set $\Sigma$ the extremal ones
are the first to be detected by linear optimization methods. For
example, in the case of the three classical truncated moment
problems on the line, they correspond to convex combinations of
point masses. Their support is identified with the zero set of
orthogonal polynomials, and the multipliers of the Dirac measures
are also computable in terms of the diagonal Pad\'e approximation
of the series:
$$ -\frac{\gamma_0}{z} -\frac{\gamma_1}{z^2}
-...-\frac{\gamma_n}{z^{n+1}}.$$ Stieltjes original memoir remains
unsurpassed for a careful analysis of this approximation scheme,
see for instance \cite{akhie1965}. A basic observation in this
direction, providing an extremal solution to Stieltjes moment
problem with the data $(\gamma_0,...,\gamma_{2n-1})$ is the
following: assuming that the Hankel matrices
\begin{align} \label{Hankel}
 \begin{pmatrix}
\gamma_0&\gamma_1&\ldots&\gamma_{n-1}\\
\gamma_1&\gamma_2&\ldots&\gamma_n\\
\vdots & & & \vdots\\
\gamma_{n-1}&\gamma_n&\ldots&\gamma_{2n-1}\\
\end{pmatrix},
&&
\begin{pmatrix}
\gamma_1&\gamma_2&\ldots&\gamma_{n}\\
\gamma_2&\gamma_3&\ldots&\gamma_{n+1}\\
\vdots & & & \vdots\\
\gamma_{n}&\gamma_{n+1}&\ldots&\gamma_{2n-1}\\
\end{pmatrix}
\end{align}
are positive definite, the one step completion
$(\gamma_0,...,\gamma_{2n-1}, \tilde{\gamma_{2n}})$ so that the
determinant
$$ \begin{vmatrix}
\gamma_0&\gamma_1&\ldots&\gamma_{n}\\
\gamma_1&\gamma_2&\ldots&\gamma_{n+1}\\
\vdots & & & \vdots\\
\gamma_{n}&\gamma_{n+1}&\ldots&\tilde{\gamma_{2n}}\\
\end{vmatrix} =0$$
vanishes, has a unique, necessarily finite, atomic solution.

Since the computation of the roots of an orthogonal polynomial is
not friendly from the numerical point of view, the search for
other special solutions of the truncated moment problem led to
adopt a statistical point of view, and consider "the most
probable" solutions, with respect to a non-linear, concave
functional. Recent applications (in particular to continuum
mechanics) use to this aim the Boltzmann-Shannon entropy, see
\cite{bende1987,borwe1991a,georg2006,junk2000,lever1996,mead1984}. The entropy
maximization method for the trigonometric moment problem stands
aside for clarity and depth in this framework, see \cite{landa1987}.

A great deal of recent work, cf. \cite{junk2000,hauck2008}, has
clarified the existence of maximum-entropy solutions, especially
in some degenerate cases. We start from there, and add a
computational/numerical analysis component to the study.

\section{Maximal entropy solutions in 1D}\label{sec:maxent1d}
For the sake of clarity
we digress and specialize the above discussion to the simplest and best-understood framework.  Namely, we discuss below
the existence and uniqueness of maximum entropy solutions to the
truncated moment problem in the case of a single variable.

\subsection{Basic properties}\label{sec:basics}

Although an abstract, fairly general treatment of the maximum
entropy method is nowadays available, see or instance
\cite{junk2000,borwe1991a}, we specialize below on an interval of the
real line. To this aim, we go back to M. Riesz' existence theorem
stated in the previous section. Namely,
 $n$ is a fixed degree and $(a,b)$ is an interval on the real line,
bounded or not. We start with the moment data
$\gamma_0,...,\gamma_n$, and seek a positive measure $\mu$ carried
by the closure of $(a,b)$ satisfying
$$\gamma_k = \int x^k d\mu(x), \ \ 0 \leq k <n,$$
and
$$ \gamma_n \geq \int x^n d\mu.$$

We search $\mu$ of the from $d\mu(x) = \exp(\lambda_0 + \lambda_1
x+...+\lambda_n x^n) dx $, assuming that the integrability
condition
$$ \int_a^b \exp(\lambda_0 + \lambda_1
x+...+\lambda_n x^n) dx <\infty$$ is assured by the choice of the
parity and sign of the leading term. For instance, in case $a=0,
b=\infty$ we must have $\lambda_p<0$ and $\lambda_{p+1} =
\lambda_{p+2} = \lambda_n = 0$; or in the case $a= -\infty, b =
\infty$ we must have $\lambda_{2p}<0$ and $\lambda_{2p+1} =
\lambda_{2p+2} = \lambda_n = 0$. We denote by $\Lambda$ (by
omitting the subscript $n$) the set of all such multipliers which
produce integrable exponentials.

The proper choice of the parameters $\lambda_k$ is made by
imposing the optimality (maximum entropy) condition:
\begin{equation}
\label{maxentropy}
 \sup \left\{ \lambda_0 \gamma_0 + ...+ \lambda_n
\gamma_n - \int_a^b \exp(\lambda_0 + \lambda_1 x+...+\lambda_n
x^n) dx \right\}
\end{equation}
where the supremum is taken over all admissible (i.e. integrable
exponential) tuples ${\lambda} = (\lambda_0,...,\lambda_n).$ Let
us similarly denote ${\gamma} = (\gamma_0,...,\gamma_n)$ and
${\mathbf x} = (1,x,x^2,...,x^n)$, where the latter is considered
as a variable point on the Veronese curve described by the list of
the first monomials.

The starting point of our discussion is the observation that the
functional
$$ L: \Lambda \longrightarrow {\R}, \ \
L({\lambda}) = {\lambda} \cdot {\gamma} -  \int_a^b  \exp [{
\lambda} \cdot {\mathbf x}] dx ,$$ is concave. Indeed, whenever
the partial derivatives are defined (for instance in  the
Euclidean interior of $\Lambda$), we have
$$ \frac{ \partial^2 L}{\partial \lambda_i \partial \lambda_j} = -
 \int_a^b x^{i+j} \exp [{ \lambda} \cdot {\mathbf x}] dx.$$
 In the above Hessian, we recognize the negative of the Hankel matrix of a
 non-atomic positive measure, whence the strict negative definiteness.
 Moreover, the inner critical points of the functional are
 given by the vanishing gradient conditions:
 $$ \frac{ \partial L}{\partial \lambda_j} = \gamma_j -
 \int_a^b x^{j} \exp [{ \lambda} \cdot {\mathbf x}] dx =0.$$

  The difficulty related to the described method lies in the
  complicated structure of the set $\Lambda$ of admissible
  multipliers. While for a bounded interval $(a,b)$ this set
  is the whole Euclidean space $\Lambda = {\R}^{n+1}$, the
  case $(a,b) = (0,\infty)$ requires:
  $$ \Lambda = [{\R}^n \times (-\infty, 0)] \cup [{\R}^{n-1} \times (-\infty,
  0)\times \{ 0 \}] \cup ...\cup [{\R} \times \{ 0 \}
  \times ... \times \{ 0\}].$$ And similarly when $(a,b) =
  (-\infty, \infty)$. On the positive side, we remark following Junk \cite{junk2000}
  that in all cases the assumption that $\gamma$ is a moment sequence implies
  $$ \lim_{\abs{\lambda} \rightarrow \infty} L(\lambda) = -
  \infty.$$

  Thus, in the bounded interval case, the optimization problem
  (\ref{maxentropy}) always has a solution, and by strict
  convexity, this is unique. Note that in this situation, the
  positivity conditions in M.\ Riesz Theorem (or equivalently
Hausdorff finite difference conditions) are necessary and
sufficient for the existence of an exponential type solution to
the truncated moment problem, see also \cite{mead1984} for a detailed
discussion.

A much more delicate analysis is required in the case of Stieltjes
moment problem $(a,b) = (0, \infty)$. For this case it is very
possible that the extremal value in problem (\ref{maxentropy}) is
attained on the boundary of the set $\Lambda$. Assume for instance
that
\begin{align*}
&\sup \left\{ \lambda_0 \gamma_0 + ...+ \lambda_n
\gamma_n - \int_a^b \exp(\lambda_0 + \lambda_1 x+...+\lambda_n
x^n) dx \right\} = \\
&\sigma_0 \gamma_0 + ...+ \sigma_n
\gamma_n - \int_a^b \exp(\sigma_0 + \lambda_1 x+...+\sigma_n x^n)
dx \,
\end{align*}
where $\sigma = (\sigma_0,...,\sigma_n) \in \Lambda
\setminus \inter \Lambda.$ That is, there exists an index
$0<p<n$ with the property
$$ \sigma_{p-1} <0 = \sigma_p = ...=\sigma_n$$
if $p>1$, or simply
$$ 0 = \sigma_1 = ...=\sigma_n$$
in case $p=1$. Anyway, then only lateral partial derivatives
$\frac{ \partial L}{\partial \lambda_j}(\sigma)$ exist for all $p
\leq j \leq n$. Since $\sigma$ is a global maximum, we infer
$$\begin{cases}
\gamma_j - \int_0^\infty x^j \exp [\sigma \cdot {\mathbf x}] dx = 
\frac{ \partial L}{\partial \lambda_j}(\sigma) \geq 0, & p \leq j \leq n,\\
\gamma_j - \int_0^\infty x^j \exp [\sigma \cdot {\mathbf x}] dx =
\frac{ \partial L}{\partial \lambda_j}(\sigma) = 0, & j<p.
\end{cases}$$

Note that above, the exponential density depends only on $p$
parameters $(\sigma_0,...,\sigma_{p-1})$, whence it is normal to
expect that only the first $p$ moments are matched.

A detailed analysis of the decision tree resulting from the above
observations goes as back as 1977 to Einbu \cite{einbu1977} and it was much
clarified in the recent works by Junk \cite{junk2000} and Hauck,
Levermore and Tits \cite{hauck2008}. We reproduce below, following Einbu
and Junk, the main phenomenon, in the form of an analysis of a one
step extension.

Suppose that, for the truncated version of Stieltjes moment
problem, the initial segment of moments
$$ (\gamma_0, \gamma_1, ..., \gamma_{n-1}) $$
is realized by the maximum entropy method, that is there is an
admissible tuple $\sigma = (\sigma_0,...,\sigma_n)$, such that
$$\gamma_j = \int_0^\infty x^j \exp [\sigma \cdot {\mathbf x}] dx,
\ 0 \leq j \leq n-1.$$ This implies that Hankel's positivity
conditions (\ref{Hankel}) hold true, and that the (lateral)
partial derivatives of the function $L(\lambda)$ vanish at
$\lambda=\sigma$.

We assume next that the extended moment sequence $(\gamma_0,
\gamma_1, ..., \gamma_{n-1}, \delta)$ is also realizable by the
maximal entropy method. Hankel's positivity conditions
$(\ref{Hankel})$ imply
$$ \delta \geq \gamma_n (\min),$$
where the bound $\gamma_n(\min)$ is a rational function of the
data $(\gamma_0, \gamma_1, ..., \gamma_{n-1})$, expressed as a
quotient of Hankel type determinants. Define
$$ \gamma_n(\max) = \int_0^\infty x^n \exp [\sigma \cdot {\mathbf x}]
dx.$$ This corresponds to the boundary point
$(\sigma_0,...,\sigma_n, 0) \in \Lambda_n$, and in addition we
know that the function $L : \Lambda_n \longrightarrow \R$,
when restricted to $\Lambda_{n-1} \times \{0 \}$, has null partial
(lateral) derivatives at $(\sigma_0,...,\sigma_n, 0)$. Assume that
$$\gamma_j = \int_0^\infty x^j \exp [\tau \cdot {\mathbf x}] dx,
\ 0 \leq j \leq n,$$ where $\tau \in \Lambda_n$. In particular
$\tau_n <0$, or $\tau_n =0$, in which case, by the uniqueness of
the maximum entropy solution $\tau = (\sigma_0,...,\sigma_n, 0)$
and $\delta = \gamma_n(\max)$.

Assume that $\tau_n<0$, so that
$$ \gamma_j - \int_0^\infty x^j \exp [\tau \cdot {\mathbf x}] dx =
\frac{ \partial L}{\partial \lambda_j}(\tau) = 0,$$ where
$\gamma_n = \delta$. Thus $\tau$ is a global maximum for the
function $L$ defined on $\Lambda_n$, and in particular $L(\tau)
\geq L(\sigma_0,...,\sigma_n, 0)$. By analyzing the restriction of
the concave function $L$ to the linear segment joining inside the
set $\Lambda_n$ the points $\tau$ and $(\sigma_0,...,\sigma_n, 0)$
we infer $\left.\frac{\partial L(t \tau + (1-t)(\sigma_0,...,\sigma_n,
0))}{\partial t}\right\vert_{t=0} \leq 0$, or in other terms
$$ \delta \leq \lambda_n(\max).$$

In conclusion, assuming that the finite moment sequence
$(\gamma_0, \gamma_1, ..., \gamma_{n-1})$ is representable by a
maximum entropy solution {\it of the same degree}, the extension
$(\gamma_0, \gamma_1, ..., \gamma_{n-1}, \gamma_n)$ has the same
property only if
$$ \gamma_n(\min) \leq \gamma_n \leq \gamma_n(\max).$$

One step further, when investigating only the solvability of
Stieltjes problem with data $(\gamma_0,...,\gamma_n)$ by the
maximum entropy solution without assumptions on the projected
string $(\gamma_0,...,\gamma_{n-1})$, the upper bound
$\gamma_n(\max)$ may become infinite, see for details \cite{junk2000}.

\subsection{Recurrence relation for the moments of an exponential
weight}\label{sec:recurrence} 
The maximum entropy method for solving the truncated
moment problem invites us to have a closer look at the full string
of moments of an exponential of a polynomial weight. We enter
below into the details of these computations, in the case of
Stieltjes moment problem.

Fix an integer $n>0$ and consider the polynomial
$$ P(x) = \sigma_0 + \sigma_1 x + ... + \sigma_n x^n,$$
with real coefficients and $\sigma_n <0$. Denote by
$$ \gamma_k = \int_0^\infty x^k \exp [P(x)] dx,\ \ k \geq 0,$$
the moments of the density $e^{P(x)} dx$. An integration by parts
yields, for all $k \geq 0$:
$$\gamma_k = \int_0^\infty x^k e^P dx = \frac{x^{k+1}}{k+1} e^P
|_0^\infty - \int_0^\infty \frac{x^{k+1}}{k+1} P' e^P dx = -
\int_0^\infty \frac{x^{k+1}}{k+1} P' e^P dx.$$

Hence, a finite difference equation relates every string of $n+1$
consecutive moments:
\begin{equation}\label{recurrence}
(k+1)\gamma_k + \sigma_1 \gamma_{k+1} + 2\sigma_2 \gamma_{k+2} +
\ldots + n \sigma_n \gamma_{k+n} = 0,\ \ k \geq 0.
\end{equation}
Since $\sigma_n \neq 0$, we obtain the following simple
observation.

\begin{lem} Let $P(x)$ be a polynomial of degree $n$, with negative leading term.
The moments of the density $e^{P(x)} dx$ are recurrently
determined by (\ref{recurrence}) from the first $n$ moments.
\end{lem}

Specifically, the linear dependence
$$ \gamma_{k+n} = - \frac{k+1}{n \sigma_n} \gamma_k -
\frac{\sigma_1}{n \sigma_n} \gamma_{k+1} - \ldots -
\frac{(n-1)\sigma_{n-1}}{n \sigma_n}\gamma_{k+n-1},$$ holds. By
changing the running index, we find for all $m>n:$
$$ \gamma_m = - \frac{m-n+1}{n \sigma_n} \gamma_{m-n} -
\frac{\sigma_1}{n \sigma_n} \gamma_{m-n+1} - \ldots -
\frac{(n-1)\sigma_{n-1}}{n \sigma_n}\gamma_{m-1}.$$

Let $M' = \max_{i=1}^{n-1} \abs{\frac{i \sigma_i}{n \sigma_n}}$ and $M
= \max (M', \abs{\frac{n-1}{n \sigma_n}})$, so that
$$ \max_{j\leq m} \abs{\gamma_j} \leq (\frac{m}{\abs{n \sigma_n}} + n M) \max_{j\leq m-1}
\abs{\gamma_j}.$$ Therefore there is a positive constant $C$ and a
positive integer $N$, such that
$$ \max_{j\leq m} \abs{\gamma_j} \leq C^m (m+N)!, \ \ m \geq 0.$$
Consequently, Stirling's formula implies
$$ \frac{\ln \max_{j\leq m} \abs{\gamma_j}}{m} \leq C + \frac{ (m+N)
(\ln (m+N) -1)}{m} + \frac{ \ln (2\pi (m+N))}{2m},$$ and in
particular
$$ \frac{\ln \max_{j\leq m} \abs{\gamma_j}}{m} \leq C' + \frac{\ln
(m+N)}{m},$$ where $C'$ is a positive constant.

In conclusion, there is a positive constant $\gamma$, such that
$$ \sum_{m=0}^\infty \frac{1}{\abs{\gamma_m}^{1/m}} \geq \sum_{m=0}^\infty \frac{1}{[\max_{j\leq m} \abs{\gamma_j}]^{1/m}}
\geq \gamma \sum_{m=0}^\infty \frac{1}{m+N} = \infty.$$ According
to Carleman's uniqueness criterion (see for instance
\cite{akhie1965}) we obtain the following result.

\begin{thm} Let $P(x)$ be a non-constant polynomial with negative leading
term. Then the moment problem with density $e^P dx$ is determined.
\end{thm}

We translate this statement for the reader who is not familiar
with the terminology: if a positive measure $\mu$ on $[0,\infty)$
has the same moments as $e^Pdx$, then $\mu = e^Pdx$.

\subsection{Existence}\label{sec:existence}
 We have seen in the previous sections that
not every truncated sequence of  moments
$(\gamma_0,\gamma_1,...,\gamma_n)$ on the semi-axis can be
achieved by the maximum entropy method, within the same degree.
That is, it is not true that there always exists an admissible polynomial $P(x)$
of degree $n$ or less, such that
\begin{equation}\label{represent}
 \gamma_k = \int_0^\infty x^k e^{P(x)} dx, \ \ k \leq n.
 \end{equation}
 
 To give the simplest example, consider the sequence
 $$ \gamma_0 = 1, \ \gamma_1 = \gamma_2 = \ldots = \gamma_n =0.$$
 Obviously, the Dirac mass $\delta_0$ has these very moments. However,
 there is no polynomial $P$, of any degree, such that
 $$ 0 = \gamma_1 = \int_0^\infty x e^{P(x)} dx.$$
 Simply because the integrand is non-negative and non-null on the interval of integration.
 
 Our study is motivated by the need to solve this pathology. In the following sections we 
 indicate a method to overcame the limitation of the maximum entropy method
 to absolutely continuous measures. Along the same lines, some recent works
 proposed different regularizations, see for instance \cite{cai2010}

\section{Single variable: Conditioning using the Cauchy transform}
\label{sec:oned}
The recent works of Junk \cite{junk2000} and Hauk, Levermore and Tits
\cite{hauck2008} clarified which positive densities $\rho$ are
appropriate for the maximum entropy reconstruction method. A thorough
analysis of the convex structure of the truncated moment set of these
distributions, e.g., extreme points, facets, was carried out in the
cited works, with significant applications for the kinetic theory of
gases. In particular, singular measures are especially poor candidates
for maximum entropy reconstruction, as seen from the example in
Section \ref{sec:existence}. In an attempt to enlarge the class of
measures for which such well established reconstruction methods work,
we propose a regularization procedure which will produce an admissible
input for the entropy optimization procedure.

The goal of our procedure is to reconstruct a possibly singular measure $\mu$ by transforming it to a continuous measure $\phi(t)dt$, whose density $\phi$ we term \emph{the phase function}. The entire measure reconstruction procedure can broken down into three steps:
\begin{enumerate}
\item regularization based on moment data of $\mu$,
\item density reconstruction (using entropy optimization) of $\phi$,
\item inversion, i.e., recovering a measure $\mu^* \approx \mu$, from point-wise knowledge of $\phi$.
\end{enumerate}
We stress here that it is not our aim to improve on the density reconstruction procedure, i.e., the entropy optimization, itself. Rather we focus on moving the density reconstruction where it can be performed with assured convergence, by inserting the regularization and inversion steps. It could be
very well possible that other density reconstruction methods, e.g., basis pursuit, wavelet-based reconstruction, could be
used instead of the maximum entropy for the general reconstruction
problem, however, we do not explore these options here.

In this section, we first focus on measures whose support lies in a one-dimensional space. In this case, the entire procedure is based on a simple idea of A.\ A.\ Markov \cite{akhiezer1962}, widely used in function theory, employing Cauchy transforms. Cauchy transforms serve as an analytic tool to study complex generating functions of the moment sequences. The regularization step is based on representation theorems for the generating function of the moment sequence, while the inversion step is grounded in Plemelj-Sokhotski formulas, which can be used to reconstruct the densities on the original domain. When the domain is one-dimensional, Plemelj-Sokhotski formulas can be formulated through a Hilbert transform, which is easily evaluated numerically. Therefore, such a reconstruction results in an algorithm that can easily be implemented in a computer code.

In Section \ref{sec:1dunbounded}, we first give the procedure for measures with arbitrary supports in $\R$, based on power, i.e., monomial, moments of the measure $\mu$ as input data. If the support of measure is contained in a compact interval, we can employ trigonometric moments instead, which are preferred numerically to power moments. The regularization procedure for compact supports is developed in Section \ref{sec:1dcompact}, and is somewhat more technical than for the unbounded case, yet the spirit is the same. Based on insights for one-dimensional domains, in Section \ref{sec:phasereg} we discuss how the procedure might be extended to measures supported in $\R^d$.

\subsection{Unbounded support}
\label{sec:1dunbounded}

Define the Cauchy transform of a measure $\mu$, with support in $\R$, as
 \begin{align} \label{eq:cauchy} \cauchy\mu(z) = \int_{\R} \frac{d\mu(x)}{x-z}.
 \end{align}
 Markov's observation is the following: assuming all integrals exist, the Cauchy transform of a {\it positive} measure on the line is of Nevanlinna class, i.e., it
 has a positive imaginary part in the upper-half plane:
 \[ \frac{ \cauchy\mu(z) - \cauchy\mu(\overline{z})}{2i} = \int_{\R} \frac{ \Im z d\mu(x)}{\abs{x-z}^2} >0, \quad \Im z >0.\]
 Hence the {\it phase}  $\Im \ln [\cauchy\mu(z)]$ is a harmonic function in
 the upper half-plane, uniformly bounded from below by zero and from above by $\pi$. The boundary values along the real line of $\Im \ln [\cauchy\mu(z)]$ produce an integrable, positive and bounded density $\phi$, satisfying:
 \begin{align}
   1 + \cauchy\mu(z) &= \exp \int_\R \frac{\phi(x) dx}{x-z}, \quad \Im z >0 \\
\shortintertext{i.e.}
   1 + \cauchy\mu(z) &= \exp \cauchy \phi(z), \quad \Im z >0,\label{eq:cauchyphaseshift}
 \end{align} 
where we slightly abuse the notation when we use $\cauchy\phi$.
The dictionary between properties of $\mu$ and density $\phi$  was established by Aronszajn and Donoghue \cite{arons1956}. The most important, of course, is the existence and boundedness of $\phi$. As $\phi$ is bounded even if $\mu$ is singular, we consider $\phi(t)dt$ to be a \emph{regularization} of $d\mu$.

Practical benefit of this expression comes from the ability to use it without knowing the closed-form expressions for measures involved. The Cauchy transform is the (complex) generating function for moments of $\mu$, i.e., its expansion at $z=\infty$ is given by 
\[
(\cauchy\mu)(z) = -\sum_{n=0}^\infty \frac{a_\mu(n)}{z^{n+1}},
\]
where $a_\mu(n) \triangleq \int_\R t^n d\mu(t)$, and $a_\phi(n)$ defined analogously.
\footnote{The non-linear transform of the moment sequence was exploited in the theory of the phase shift of perturbed spectra in quantum mechanics, see \cite{birma1992,polto1998}.} Solving for $\cauchy \phi$ and using the series expansion $\ln(1 + z) = -\sum_{n=1}^\infty (-1)^n z^n / n$ yields the following equality between power series:
\begin{align*}
  \cauchy \phi(z) = \ln[ 1 + \cauchy \mu(z)]
  &= - \sum_{k=1}^\infty \frac{1}{k} [-\cauchy \mu(z)]^k \\
 \sum_{n=0}^\infty \frac{a_\phi(n)}{z^{n+1}} &= \sum_{k=1}^\infty \frac{1}{k}\left[ \sum_{n=0}^\infty \frac{a_\mu(n)}{z^{n+1}} \right]^k.
\end{align*}
The Miller-Nakos Theorem \cite{nakos1993}, whose complete proof we bring in the \ref{sec:miller}, gives a recursion for evaluation of moments $a_\phi(n)$ from moments $a_\mu(k)$ for $k=0, \dots, n$,
\begin{align}
  a_\phi(N) = \sum_{k=1}^N \frac{1}{k}[S_N(z)]^k_N,\label{eq:1dmomentconversion}
\end{align}
where $[S_N(z)]^k_N$ indicates the coefficient next to $z^{-(N+1)}$, in the $k$-th power of the truncation  $S_N(z) =  \sum_{n=0}^N a_\mu(n) z^{-(n+1)}$ of the generating power series. Such a triangular property is essential for practical problems: we will typically have access only to truncated moment data and we do not wish to establish any a priori ansatz, especially not $a_\mu(n) = 0$ for $n > N$.

At this point, we have set up moment data such that most density reconstruction procedures apply: density $\phi$ is bounded and compactly supported, making it possible to reconstruct it using entropy optimization described in Section \ref{sec:maxent1d}. Such a procedure produces an approximant 
\begin{align}
  \phi^\ast(x) = \exp \sum_{k=0}^N \alpha_k x^k\label{eq:entropy-solution}  
\end{align}
that converges to density $\phi$ as the number of available moments $N$ increases.

The inversion step describes how the point-wise knowledge of approximant $\phi^\ast \approx \phi$ is used to compute an absolutely continuous measure $\mu^\ast$ that approximates the original measure $\mu$. In this paper we do not claim to obtain  quantitative convergence results on $\mu^\ast \to \mu$, especially when $\mu$ is a singular measure, however, we stress that, for singular measures, a classical reconstruction procedure like entropy optimization might not produce any results. Therefore, we view our results in this paper as a starting point for further investigations of approximation of singular measures.

To a smooth entropy optimizer  $\phi^\ast$ corresponds a measure $\mu^\ast$ with a density $\rho = d\mu^\ast/dx$, which we use to approximate the original measure $\mu$. The lynchpin of the inversion procedure, i.e., evaluation of $\rho$ from knowledge of $\phi^\ast$, is the existence of boundary limits $\lim_{\epsilon \to 0}\cauchy \mu(x \pm i\epsilon)$, for $\epsilon > 0$. The limits exist independently pointwise, and, assuming that $\rho \in L^1(\R)$ is of H\"older-class, the Plemelj-Sokhotski formulas, e.g., \cite[][\S 14.11]{henrici1977vol3} or \cite[][\S 3.7]{king2009v1}, establish that it is possible to evaluate $\rho$ pointwise from Cauchy transforms of $\mu$ as
\[
  \rho(x) = \frac{1}{2 \pi i} \lim_{\epsilon \downarrow 0}[ \cauchy\mu^\ast(x + i\epsilon) - \cauchy\mu^\ast(x - i\epsilon)].
\]
As limits exist independently, and $\exp$ is analytic, we can formulate them in terms of analogous limits for Cauchy transforms of the phase function $\cauchy\phi^\ast$, i.e. by \eqref{eq:cauchyphaseshift},
\begin{align*}
  \rho(x) &= \frac{1}{2 \pi i} \lim_{\epsilon \downarrow 0}[ \exp \cauchy\phi^\ast(x + i\epsilon) - \exp \cauchy\phi^\ast(x - i\epsilon)] \\
  &= \frac{1}{2 \pi i} \left[ \exp \lim_{\epsilon \downarrow 0} \cauchy\phi^\ast(x +
  i\epsilon) - \exp \lim_{\epsilon \downarrow 0} \cauchy\phi^\ast(x - i\epsilon)\right].
\end{align*}
Moreover, the Plemelj-Sokhotski formulas provide explicit expressions for each limit:
\begin{align}\label{eq:ps-realline}
  \lim_{\epsilon \downarrow 0} \frac{1}{2\pi i}\cauchy \phi^\ast(x \pm
  i\epsilon) &= \pm \frac{1}{2} \phi^\ast(x) + \frac{i}{2} \hilb
\phi^\ast(x), \\
\shortintertext{where the Hilbert transform is}  \label{eq:hilbert-realline}
\hilb \phi^\ast(x) &=
\frac{1}{\pi}\dashint \frac{\phi^\ast(t)dt}{t - x}, 
\end{align}
It follows then that the expression for $\rho$ is given by:
\begin{align}
  \label{eq:onedinverse}
  \rho(x) = \frac{1}{\pi} \exp\left[-\pi \hilb\phi^\ast(x) \right] \sin \pi\phi^\ast(x).
\end{align}
This formula connects density $\rho = d\mu/dx$ with the phase density function $\phi$, or, in the case of moment closure by entropy optimization, a smooth approximant $\phi^\ast$ to the phase density function $\phi$. 

The formula \eqref{eq:onedinverse} is numerically practical: the 
entropy optimization provides us with a closed formula for $\phi^\ast$, while its Hilbert transform is easily numerically evaluated via the Fast Fourier Transform algorithm. Therefore, here we have obtained a practical inversion formula for approximating a singular measure $\mu$ via an absolutely continuous measure $\mu^\ast$ with density  $\rho = d\mu^\ast/dx$.

\subsection{Compact support}
\label{sec:1dcompact}

When the measure $\mu$ is supported on a known compact interval, we can use trigonometric moments, instead of power moments, in the process given above. The resulting process is more numerically robust, as trigonometric functions are orthonormal and bounded as a family, unlike the family of monomials on an arbitrary interval.

Let $\mu$ be a measure on the interval $\Delta = [-\pi,\pi)$ that induces the measure $\breve \mu$ on the boundary $\partial \D$ of the unit disk $\D \subset \C$. A known relation is then $d\mu(\theta) = -i \bar \zeta d\breve \mu(\zeta)$, for $\zeta = e^{i\theta} \in \partial \D$. Define the circular Cauchy transformation
\[
\mathcal K \mu(z) \triangleq \frac{1}{2\pi}\int_{\partial \D} \frac{d\breve \mu(\zeta)}{\zeta - z},
\]
for $z \not \in \partial \D$. Using the equivalent arc-length formulation clarifies the difference between
\[
\mathcal K \mu(z) = \frac{i}{2\pi} \int_{-\pi}^\pi \frac{ d\mu(\theta) }{1 - e^{-i\theta} z}
\]
 and the Cauchy transform on the line $\mathcal C \mu(z) = \int_{-\pi}^\pi d\mu(x)/(x-z)$, cf. \eqref{eq:cauchy}.

The function $\mathcal K \mu$ is holomorphic inside $\inter \D$, $\mathcal K\mu \in \mathcal O(\D)$, where it has the Taylor expansion
\begin{align}
\mathcal K \mu(z) &= i \sum_{k=0}^\infty \tau_\mu(k) z^k,\notag\\
\shortintertext{with complex trigonometric moments}
\tau_\mu(k) &\triangleq \frac{1}{2\pi}\int_{\partial \D} \bar \zeta^{k} \frac{d\breve\mu(\zeta)}{i \zeta}\label{eq:trigmom}
\end{align}
serving as coefficients.

The imaginary part $\Im \mathcal K \mu(z)$, is positive for positive measures, as the imaginary part of the kernel is
\begin{align*}
  \frac{1}{4\pi i}\left(\frac{i}{1 - e^{-i\theta}z} - \frac{-i}{1 - e^{i\theta}\bar z}\right) = \frac{1 - \Re ( e^{-i\theta}z ) }{2\pi\abs{1 - e^{-i\theta}z }^2 },
\end{align*}
and $\abs{z e^{-i\theta}} < 1$ when $z \in \inter \D$. Consequently, the argument of $\mathcal K\mu(z)$, with the appropriately chosen branch of the logarithm, 
\begin{align}
  F(z) \triangleq -i \ln \mathcal K \mu(z) \in \mathcal O(\D),\label{eq:circcauchyangle}
\end{align}
is of Caratheodory class: it is a positive function, with a bounded real part $\Re F(z) \in [0,\pi]$, which corresponds to the bounded angle of $\mathcal K\mu(z)$.

The following classical theorem allows us to obtain a representation of Caratheodory class functions in terms of bounded densities on a circle \citep[e.g.][\S 12.10]{henrici1977vol2}:
\begin{thm}[Riesz-Herglotz] \label{thm:RieszHerglotz} Let $F \in \mathcal O(\D)$ be such that $\Re F(z) \in [0, c]$ for some fixed $c > 0$. Then there exists a function $\breve \phi \in L^1(\partial \D)$ for which 
  \begin{align*}
    F(z) &= i\Im F(0) + \mathcal P \phi(z), \\\shortintertext{where $\breve \phi(\zeta) \in [0, c]$ pointwise and }
    \mathcal P \phi (z) &\triangleq \frac{1}{2\pi}\int_{-\pi}^\pi \frac{e^{i\theta} +
      z}{e^{i\theta} - z} \phi(\theta)d\theta,
  \end{align*}
is the Poisson integral of $\phi(\theta) \equiv \breve\phi(e^{i\theta})$.
\end{thm}

The Poisson integral $\mathcal P \phi$ can be rewritten in terms of the circular Cauchy transform $\mathcal K \phi$:
\[
\mathcal P \phi(z) = \frac{1}{2\pi}\int_{\partial \D} \frac{\zeta + z}{\zeta - z} \frac{\phi(\zeta) d\zeta}{i\zeta} = - \tau_\phi(0) - i 2 \mathcal K\phi(z),
\]
where, again, $\tau_\phi(k)$ are trigonometric moments of $\phi$, defined analogously to \eqref{eq:trigmom}. The Riesz-Herglotz formula then reads 
\[
F(z) = -i 2 \mathcal K\phi(z) - \tau_\phi(0) + i\Im F(0).
\]
We can compute the constants in the formula by evaluating it at $z=0$ and comparing it to evaluation of the definition \eqref{eq:circcauchyangle} at the same point:
\begin{align*}
  F(0) = i\Im F(0) - \tau_\phi(0) -i 2 [i \tau_\phi(0) ] &= \tau_\phi(0) + i \Im F(0) \\
  F(0) = -i \ln i \tau_\mu(0) &= \frac{\pi}{2} - i \ln \tau_\mu(0),
\end{align*}
concluding that 
\[
\tau_\phi(0) = \frac{\pi}{2}, \quad \Im F(0) = -\ln \tau_\mu(0). 
\]
Substituting these constants into the Riesz-Herglotz formula, and using the definition of $F(z)$, we obtain the exponential representation of $\mathcal K \mu(z)$:
\begin{align}
\mathcal K\mu(z) = -i \tau_\mu(0) \exp[ 2\mathcal K\phi(z) ].\label{eq:phasecircle}
\end{align}

To compute the moments of $\phi$, we relate the Taylor expansions of the functions above, and use $\tau_\phi(0) = \pi/2$ to obtain
\[
1 + \sum_{n=1}^\infty \hat\tau_\mu(n)z^n = \exp\left[ 2 i \sum_{n=1}^\infty \tau_\phi(n)z^n
\right],\]
where $\hat \tau_\mu(n) \triangleq \tau_\mu(n) / \tau_\mu(0)$. As before, we use the expansion $\ln(1 + z) = -\sum_{n=1}^\infty (-1)^n z^n / n$ to relate the series through expression
\[
\sum_{k=1}^\infty \tau_\phi(k) z^k  = \frac{i}{2} \sum_{k=1}^\infty \frac{(-1)^k}{k} \left[ \sum_{n=1}^\infty \hat \tau_\mu(n) z^n \right]^k.
\]
A finite number $M$ of trigonometric moments $\tau_\phi(k)$ can then be computed using the Miller-Nakos algorithm (see \ref{sec:miller}) when $M$ moments $\tau_\mu(k)$ are known.

To invert the procedure, we assume that to approximate $\phi$, we are given a smooth density $\phi^\ast: [-\pi,\pi] \to \R$, which corresponds to a continuous $\mu^\ast$ with density $\rho: [-\pi,\pi] \to \R$, i.e., $d\mu^\ast(\theta) = \rho(\theta) d\theta$. 
Density $\rho$ can be evaluated point-wise using Plemelj-Sokhotski formulas (see Dynkin's chapter, section \S 6 in \cite{encyc15}), which evaluate non-tangential limits $\ilim_{\xi \to z} \mathcal K\mu^\ast(\xi)$ and $\elim_{\xi \to z} \mathcal K\mu^\ast(\xi)$ at $z \in \partial\D$, with the argument in domains $\xi \in \inter \D$ and $\xi \in \C / \D$, respectively. The A.\ Calder\'on's theorem asserts existence of such limits for $\breve \phi^\ast \in L^1(\partial \D)$.

Due to analyticity of $\exp$ in \eqref{eq:phasecircle}, we can evaluate the non-tangential limits of $\mathcal K \mu^\ast$, in terms of non-tangential limits $\mathcal K \phi^\ast$ of $\breve\phi^\ast(\zeta) d\zeta$\footnote{We slightly abuse the notation when we use $\mathcal K\phi^\ast$.}
\begin{align*}
  \ilim_{\xi \to z} \mathcal K\mu^\ast(\xi) &= -i\tau_\mu(0) \exp[ 2 \ilim_{\xi \to z} \mathcal K\phi^\ast(\xi)] \\
  \elim_{\xi \to z} \mathcal K\mu^\ast(\xi) &= -i\tau_\mu(0) \exp[ 2 \elim_{\xi \to z} \mathcal K\phi^\ast(\xi)] 
\end{align*}

Privalov's Lemma establishes that the non-tangential limits satisfy Plemelj-Sokhotski formulas for $z \in \partial \D$:
\begin{align*}
  \ilim_{\xi \to z} \mathcal K\phi^\ast(\xi) &= \mathcal Q\phi^\ast(z) + \frac{i}{2} \breve\phi^\ast(z), \\
  \elim_{\xi \to z} \mathcal K\phi^\ast(\xi) &= \mathcal Q\phi^\ast(z) - \frac{i}{2} \breve\phi^\ast(z), \\ \shortintertext{where $\mathcal Q$ indicates the singular integral}
  \mathcal Q\phi^\ast(z) &\triangleq \frac{1}{2\pi}\dashint_{\partial \D} \frac{\breve\phi^\ast(\zeta)d\zeta}{\zeta - z},
\end{align*}
with analogous expressions holding for $\mathcal K\mu^\ast(z)$, with density $\rho(\theta) = d\mu^\ast/d\theta$ instead of $\phi^\ast$. Therefore, to evaluate $\breve\rho(z)$ on $\partial \D$ we seek the difference between the non-tangential limits:
\begin{align*}
\breve\rho(z) &= -\tau_\mu(0) \left\{\exp[ 2 \ilim_{\xi \to z} \mathcal K\phi^\ast(\xi)] - \exp[ 2 \elim_{\xi \to z} \mathcal K\phi^\ast(\xi)]\right\}\\
&= -i 2\tau_\mu(0) \exp[ 2\mathcal Q\phi^\ast(z) ] \sin \phi^\ast(z)
\end{align*}
The singular integral $\mathcal Q \phi^\ast(z)$ can be evaluated on $z \equiv e^{i\theta}$ using the circular Hilbert transform of $\phi^\ast$:
\begin{align*}
  \mathcal Q\phi^\ast(z) &= \frac{1}{2\pi}\dashint_{\partial \D} \frac{\breve\phi^\ast(\zeta)d\zeta}{\zeta - z} =
  \frac{i}{4\pi} \dashint_{\partial \D} \frac{\zeta+z}{\zeta-z} \frac{\breve\phi^\ast(\zeta)d\zeta}{i\zeta} + \frac{i}{2}\tau_\phi(0) \\
  &= \frac{i}{4\pi} \dashint_{-\pi}^\pi \frac{e^{i\sigma} + e^{i\theta}}{e^{i\sigma} - e^{i\theta}} \phi^\ast(\sigma) d\sigma + \frac{i\pi}{4}\\
  &= \frac{1}{4\pi} \dashint_{-\pi}^\pi \cot \frac{\sigma - \theta}{2} \phi^\ast(\sigma)d\sigma + \frac{i \pi}{4} \\&= \frac{1}{2} \hilb \phi^\ast(\theta) + \frac{i\pi}{4},
\end{align*}
where $z \equiv e^{i\theta}$. The circular Hilbert transform is, by one convention,
\[
\hilb \phi^\ast(\theta) \triangleq \frac{1}{2\pi} \dashint_{-\pi}^\pi \cot\frac{\sigma - \theta}{2} \phi^\ast(\sigma)d\sigma.
\]
Finally, substituting this expression into $\breve \rho(\zeta) \equiv \rho(\theta)$, we get the evaluation of the density $\rho(\theta)$ as
\[
\rho(\theta) = 2\tau_\mu(0) \exp[ \hilb \phi^\ast(\theta)] \sin \phi^\ast(\theta)
\]
Practically, Hilbert transform is easily evaluated on a fixed grid using numerical Fourier Transform, e.g., FFT, so this formula can be employed when we have access to the evaluation of $\phi^\ast$ on a fixed grid in $[-\pi,\pi]$.

\section{Several variables: Conditioning using complex Fantappi\`e transforms}\label{sec:phasereg} 

The reminder of this paper deals with generalization of the regularization procedure to measures of several variables. To do so, we will replace Cauchy transform with a Fantappi\`e transform, which is usually defined as a real integral transform, relying on two sources: the harmonic analysis on tube domains over convex cones \cite{Faraut1994} and the complete monotonicity results, \`a la Bernstein, characterizing the Laplace and Fantappi\`e transforms of positive measures on convex cones \cite{Henkin1990}. For expository material on Fantappi\`e transform, see \citep[][\S 3]{ander2004}.

We propose three different ways to complexify it, in order to use the general Riesz-Herglotz representation theory and obtain analogs to the phase function $\phi$. The choice of the complexification procedure is based on a trade-off: presently we are able to obtain either theoretically general results with little practical value, or practically useful results which do not allow for as much theoretical breadth. We expect that the future research will bridge this gap between theory and computation.

Take a solid, acute, closed convex cone $\Gamma \subset \R^d$, and its associated \emph{polar cone}
 \[ \Gamma^\ast \triangleq \{ x \in \R^d; \ \ \omega \cdot x \geq 0, \  \omega \in \Gamma \}. \] The cone $\Gamma^\ast$ will carry the support of the measure $\mu$, while $\Gamma$ will play the role of the ``frequency parameter'', to use the language of signal processing and applied Fourier/Laplace analysis.  The following characterization of the real-valued Fantappi\`e transform is due to \cite{Henkin1990}:
 \begin{thm} [Henkin-Shananin] A function $\Phi : (0,\infty) \times \Gamma \to \R$ is the Fantappi\`e transform
   \begin{align}\label{eq:fantappie-real}
     \Phi (\omega_0,\omega) = \int_{\Gamma^\ast}
     \frac{d\mu(x)}{\omega_0 + \omega \cdot x}, 
   \end{align}
 of a positive measure $\mu$ supported by $\Gamma^\ast$ if and only if $\Phi$ is
\begin{inparaenum}[\upshape(\itshape i\upshape)]
 \item continuous,
 \item completely monotonic\footnote{ A function $\Phi$ is {\it completely monotonic} if it satisfies inequalities
 $ (-1)^k D_{\xi_1} ...D_{\xi_k} \Phi(p) \geq 0$,
 $\forall k \geq 0, p \in \inter \Gamma$, where $D_{\xi_k}$ are partial derivatives along coordinates $\xi_1,...,\xi_k \in \Gamma$. }, and
\item homogeneous of degree $-1$, i.e.,
 \[ \Phi(\lambda \omega_0, \lambda \omega) = \lambda^{-1} \Phi(\omega_0, \omega), \ \ \omega_0>0, \omega \in \Gamma, \lambda >0.\]
\end{inparaenum}
\end{thm}

To extend the Fantappi\`e transform to complex domains, one has a choice of complexifying the offset parameter $\omega_0$, the normal parameter $\omega$, or both. We start with the full generality in Section \ref{sec:tube}, complexifying both $\omega_0$ and $\omega$ to tube domains, and develop the full regularization procedure, at a cost of providing no inversion formulas. Next, we constrain $\omega_0 = 1$ in Section \ref{sec:restrtube}, obtaining a restricted tube domain, to provide some practical regularization formulas for measures on familiar compact domains, which stand in direct analogy to one-dimensional problem. A future research direction could explore Clifford algebras as a setting for generalization of the Plemelj-Sokhotski formulas, which were  used to complete the inversion process in the one-dimensional case in Section \ref{sec:oned}. A shorter, perhaps more immediately practical, procedure is given in Section \ref{sec:partial-fantappie}, where we treat $\omega$ as a parameter, complexifying only $\omega_0$. As a consequence, we recover a single variable procedure at each value of $\omega$, which we term \emph{the partial Fantappi\`e transform}. The family of solutions, parametrized by $\omega$, could be used in a tomographic procedure to recover an approximant to the original measure $\mu$.
 
\subsection{Unbounded supports and tube domains}\label{sec:tube}

The Fantappi\`e integral transform extends analytically to the complex domain:
 \[ \Phi(u_0, u) \triangleq \int_{\Gamma^\ast} \frac{d\mu(x)}{u_0 + u \cdot x}, \quad \Re u \in \inter \Gamma, \Re u_0 > 0,\]
retaining, by definition, homogeneity of degree $-1$ in the complex argument $(u_0,u) \in \C \times \C^d$. Let $\Omega = (0,\infty) \times \inter \Gamma \subset \R^{d+1}$ be the interior of the domain of continuity for the real Fantappi\`e transform. 

The associated \emph{tube domain} is the set $T_\Omega = \Sigma + i \Omega$, where $\Sigma = \R^{d+1}$. Due to the different role played by the first axis, we denote elements by $(z_0,z) = (\sigma_0 + i\omega_0, \sigma+i\omega) \in T_\Omega$. With this notation\footnote{Such a choice conforms with the existing conventions of harmonic analysis, and we apologize in advance for all the resulting multiplicative imaginary unities. Fantappi\`e transforms in the later sections simplify this convention.} the domain of analyticity can be written as $-iT_\Omega$, i.e., $\Phi \in \mathcal O(-iT_\Omega)$. For clarity, we will use $u$ to denote elements of $-iT_\Omega$, and $z$ for elements of $T_\Omega$, with the obvious change of coordinates $u = -iz = \omega - i \sigma$, when $z = \sigma + i \omega$.

Notice that when  $(u_0, u) \in -iT_\Omega$, 
\[
\Re \Phi(u_0, u) = \int_{\Gamma^\ast} \frac{\omega_0 + \omega \cdot x}{\abs{u_0 + u \cdot x}^2}d\mu(x) > 0,
\]
following from the definition of the polar cone $\Gamma^{\ast}$. Therefore, function $i\Phi(u_0,u)$ is analytic, and has a positive imaginary part. It follows that its complex phase $\ln i\Phi(u_0, u)$ is well defined on $-iT_\Omega$, with the property
\[
\Im \ln i\Phi(u_0, u) \in (0,\pi),\quad (u_0, u) \in -i T_\Omega.
\]
Converting this expression to the tube domain, define $F(z_0, z)$ on the tube domain $T_\Omega$ by setting
\begin{align}
  F(z_0, z) &\triangleq -i \ln i\Phi( -i z_0, -i z),\notag
  \shortintertext{with a further simplification} 
  F(z_0,z) &= -i \ln[-\Phi(z_0, z)],\label{eq:fantappiephase1} 
  \shortintertext{due to homogeneity of
    $\Phi$, or}
  \Phi(z_0, z) &= - \exp iF(z_0, z).\label{eq:fantappiephase2} 
\end{align}

Defined this way, the function $F(z_0, z)$ is analytic on $T_\Omega$, and satisfies $\Re F(z_0, z) \in [0,\pi]$. 

These properties make it possible to reveal the structure of functions $F$ via a straightforward generalization of Riesz-Herglotz formula for analytic functions of positive real part (see Theorem \ref{thm:RieszHerglotz}), which is our next goal.

Let $\Omega \subset \R^d$ be an open, acute and solid convex cone, with associated
tube domain $T_\Omega = \R^d +i\Omega$. The Hardy space  $H^2(T_\Omega)$ is defined as the space of analytic functions
$F : T_\Omega \longrightarrow \C$, such that
\[
\abs{F}^2 = \sup_{\omega \in \Omega} \int_{\R^d} \abs{F(\sigma+i\omega)}^2 d\sigma < \infty.
\] 
By a celebrated theorem of Paley and Wiener, $H^2(T_\Omega)$
is the space of Fourier-Laplace transforms of square integrable functions defined on the polar cone. 

The following result characterizes real Fourier-Laplace transforms  \cite{Henkin1990}:
 \begin{thm}[Bernstein, Bochner, Gilbert] A function $F : \Omega \longrightarrow \R$ is the Laplace transform 
 \[ F(\omega) = \int_{\Omega^\ast} e^{-\omega \cdot x} d\mu(x),\]
 of a positive measure
 $\mu$ supported by $\Omega^\ast$ if and only if $F$ is continuous on $\Omega$ and of class $C^\infty$ and completely monotonic in the interior $\inter \Omega$.
 \end{thm}
The extension from the cone $\Omega$ to tube domain $T_\Omega$, for
$f \in L^2(\Omega^\ast, dx)$, is given by
\[ 
F(z) = \frac{1}{\sqrt{(2\pi)^d}} \int_{\Omega^\ast} e^{i z\cdot x} f(x) dx,
\]
for $z \in T_\Omega$. Function $F$ then belongs to $H^2(T_\Omega)$ and the map $f \mapsto F$ is an isometric isomorphism between the two Hilbert spaces.
For a proof and an overview of the theory of Hardy spaces on tube domains see \cite{Faraut1994}.

The reproducing kernel of the Hardy space, also known as {\it Szeg\"o's kernel} is
\[ S(z,w) = \frac{1}{(2\pi)^d} \int_{\Omega^\ast} e^{i (z-\overline{w})\cdot x} dx,\]
for $z,w \in T_\Omega$. Remark the homogeneity property:
$$ S(\lambda z, \lambda w) = \lambda^{-d} S(z,w), \ \ \ \lambda>0.$$
The reproducing property has the following form: if $F \in H^2(T_\Omega)$, then the boundary limits, still denoted by F, satisfy
$$ F(\sigma) = \lim_{\omega \rightarrow 0} F(\sigma+i\omega),$$
where the limit exists in $L^2(\R^d)$ and 
$$ F(z) = \int_{\R^d} S(z,\sigma) F(\sigma) d\sigma.$$ 

We focus next on functions $F \in A(T_\Omega)$ which are analytic in $T_\Omega$ and uniformly bounded and continuous
on its closure. Then the function $z \mapsto S(z,w) F(z)$ belongs to $H^2(T_\Omega)$, for every $w \in T_\Omega$, and
$$ S(z, w) F(z) = \int_{\R^d} S(z,\sigma)S(\sigma,w) F(\sigma) d\sigma,$$
and by complex conjugation
$$ S(z,w) \overline{F(w)} = \int_{\R^d} S(z,\sigma)S(\sigma,w) \overline{F(\sigma)} d\sigma.$$
By adding the two identities we obtain
\[ S(z,w) \frac{F(z) +\overline{F(w)}}{2} = \int_{\R^d} S(z,\sigma) S(\sigma,w) \Re F(\sigma) d\sigma.\]

The restriction to the diagonal of the above formula yields
$$ \Re F(z) = \int_{\R^d} P(z,\sigma) \Re F(\sigma) d\sigma,$$
where 
$$ P(z,\sigma) = \frac{\abs{S(z,\sigma)}^2}{S(z,z)}, \ \  z \in T_\Omega, \ \sigma \in \R^d,$$
is Poisson's kernel. Again, see \cite{Faraut1994} for full details.

Fix a point $\alpha \in i\Omega$, and subtract the identities
$$S(z,\alpha)[F(z) + \overline{F(\alpha)}] = \int_{\R^d} 2 {S(z,\sigma)S(\sigma,\alpha)} \Re F(\sigma) d\sigma,$$
$$S(z,\alpha) \frac{F(\alpha)+\overline{F(\alpha)}}{2} = \int_{\R^d} \frac{\abs{S(\sigma,\alpha)}^2 S(z,\alpha)}{S(\alpha, \alpha)} \Re F(\sigma) d\sigma.$$
One finds
$$ S(z,\alpha)[F(z) - i \Im F(\alpha)] =  \int_{\R^d} \left[ 2 S(z,\sigma)S(\sigma,\alpha)  -  \frac{\abs{S(\sigma,\alpha)}^2 S(z,\alpha)}{S(\alpha, \alpha)}\right] \Re F(\sigma) d\sigma,$$
a tube domain analogue of the Schwarz formula, which relates the values of an analytic function (in the disk) to the boundary values of its real part.
We call, by way of natural analogy to the similar integral kernel for the disk,
$$ H(z, w; \alpha) = 2 \frac{S(z,w)S(w,\alpha)}{S(z,\alpha)}  -  \frac{\abs{S(w,\alpha)}^2 }{S(\alpha, \alpha)},$$
the {\it Herglotz kernel} associated to the tube domain $T_\Omega$, with $z,w \in T_\Omega$, $\alpha \in i\Omega$.

In particular cases one can obtain from here an integral representation of all
analytic functions in $T_\Omega$ possessing positive real part, what it is customarily called the Riesz-Herglotz formula,
see for details \cite{aizenberg1976}. Fortunately, our aim is more modest, having to deal only with the Fantappi\`e transforms of measures
appearing in the  previous section.

From now on, we return to the convex cone $\Omega = (0,\infty) \times \inter \Gamma \subset \R^{d+1}$ appearing in definition of the real Fantappi\`e transform \eqref{eq:fantappie-real}.
Specifically, the analytic function
$$ F(z_0,z) = -i \ln [ -\Phi (z_0,z)], \ \ (z_0,z) \in T_\Omega = \R^{d+1} + i\Omega$$
satisfies
$$ 0 < \Re F(z_0,z) <\pi, \ \ (z_0,z) \in T_\Omega$$
and, due to the homogeneity of $\Phi$ it growths logarithmically along rays contained in $T_\Omega$:
$$ \abs{F(\lambda z_0, \lambda z)} \leq \abs{F(z_0,z)} + \abs{\ln \lambda}, \ \ (z_0,z) \in T_\Omega, \lambda>0.$$
Fix a point $\alpha \in i\Omega$ and consider the translated functions
\[ F_\epsilon (z_0,z) = F[ (z_0,z) + \epsilon \alpha],\]
so that they are analytic on the closure of $T_\Omega$ and of logarithmic growth along rays.
Due to the homogeneity of degree $-(d+1)$ of the reproducing kernel $S$ of $T_\Omega$, we deduce that
for every $\epsilon>0$, the function $F_\epsilon S \in H^2(T_\Omega)$, and in view of the computations above:
$$  F_\epsilon (\zeta) = i \Im F_\epsilon (\alpha) + \int_{\R^{d+1}} H(\zeta, \sigma; \alpha) \Re F_\epsilon (\sigma) d\sigma, \ \ \zeta = (z_0,z) \in T_\Omega, \ S(\zeta,\alpha) \neq 0.$$
Note that $\Re F_\epsilon \in (0,\pi)$ on $T_\Omega$, as restriction of the original function to a subset of the tube domain. By passing with $\epsilon$ to zero and to a weak-* limit in $L^\infty (\R^{d+1})$ we obtain a function $\phi \in L^\infty(\R^{d+1}), \ 0 \leq \phi \leq \pi, a.e.,$  representing $F$ as follows:
$$ F(\zeta) = i \Im F(\alpha) + \int_{\R^{d+1}} H(\zeta, \sigma; \alpha) \phi(\sigma) d\sigma, \ \ \zeta = (z_0,z) \in T_\Omega, \ \ S(\zeta, \alpha) \neq 0.$$
In all instances of interest Szeg\"o's kernel $S(\zeta,\alpha)$ does not vanish at all points
$\zeta, \alpha \in T_\Omega$, producing a genuine integral representation of $F$.

We collect the above remarks into a formal statement.
\begin{prop}\label{exp-transform}
Let $\Gamma \subset \R^d$ be a closed, solid and acute convex cone and let $\mu$ be a finite mass positive measure 
supported by the polar cone $\Gamma^\ast$.
The Fantappi\`e transform of the measure $\mu$ admits the exponential representation:
$$ \int_{\Gamma^\ast} \frac{d\mu(x)}{z_0 + z \cdot x} = -\exp [i F(z_0,z)],\ \ (z_0,z) \in T_\Omega,$$
where $\Omega = (0,\infty) \times \inter \Gamma.$
In its turn, the analytic function $F$ admits the integral representation
$$  F(\zeta) = iC + \int_{\R^{d+1}} H(\zeta, \sigma; \alpha) \phi(\sigma) d\sigma, \ \ \zeta = (z_0,z) \in T_\Omega,$$
where $\phi: \R^{d+1} \longrightarrow [0,\pi]$ is a measurable function and $C \in \R$ is a real constant.
\end{prop}

We assume in the above statement that $\alpha \in i\Omega$ is fixed and $S(\zeta, \alpha) \neq 0.$ In this sense the, possibly singular, measure $\mu$ is regularized by the absolutely continuous measure $\phi(\sigma) d\sigma$. The integral kernel  $H(\zeta, \sigma; \alpha)$ is of little practical use in its full generality. However, in particular cases, to be discussed in the rest of the article, the preceding
Markov exponential transform regularization becomes more accessible.

One question which naturally arises in the above statement is: is it possible to characterize the bounded densities $\phi$
appearing in the integral representation of an analytic function $F \in {\mathcal O}(T_\Omega)$ that satisfies $0 \leq \Re F \leq 1?$
The answer is yes, but the conditions imposed on $\phi$ are not friendly. They were discovered a long time ago, 
in the case of the polydisk \cite{koranyi1965} and more general symmetric domains \cite{aizenberg1976}. We merely indicate these conditions
in the case of non-vanishing Szeg\"o kernel and sketch the proof.

\begin{prop} Let $\Omega \subset \R^d$ be an open, acute, solid cone, let $\alpha \in i\Omega$ and assume
$S(\zeta,\xi) \neq 0, \ \zeta, \xi \in T_\Omega$. An element $\phi \in L^\infty(\R^d)$ is the phase of an analytic function $F \in {\mathcal O}(T_\Omega), \ \ 0 \leq \Re F \leq 1:$ 
\begin{equation}\label{RH-sufficient}
F(\zeta) =iC +  \int_{\R^{d}} H(\zeta, u; \alpha) \phi(u) du, \ \ \zeta \in T_\Omega,
\end{equation}
where $C \in \R$, if and only if $0 \leq \phi \leq 1, a.e.,$ and the ``moment conditions''
$$ \int_{\R^d} [ H(\zeta, \sigma; \alpha) +  H(\alpha, \sigma; \xi) -  
H(\zeta, \sigma; \xi)] \phi(\sigma) d\sigma = 0, \ \ \zeta, \xi \in T_\Omega$$
hold.
\end{prop} 

\begin{proof} In order to prove the non-trivial implication, let $\phi \in L^\infty(\R^d), 0 \leq \phi \leq 1, a.e.,$ and define the function
$F(\zeta)$ by formula (\ref{RH-sufficient}). In view of the definition of Herglotz' kernel, by taking $\zeta = \alpha$ we find
\[ F(\alpha) -iC = \int \frac{\abs{S(u,\alpha)}^2}{S(\alpha,\alpha)} \phi(u) du,\] whence by addition
\[ F(\zeta) + \overline{F(\alpha)} = 2 \int \frac{S(\zeta,u)S(u,\alpha)}{S(\zeta,\alpha)} \phi(u) du.\]
Write this formula for $F(\xi) + \overline{F(\alpha)}$, too. According to the moment conditions in the statement, we find again by addition
$$ F(\zeta) + \overline{F(\xi)} = 2 \int \frac{S(\zeta,u)S(u,\xi)}{S(\zeta,\xi)} \phi(u) du, \ \ \zeta, \xi \in T_\Omega.$$

In particular, $\Re F$ is the Poisson's transform of $\phi$, and therefore $0 \leq \Re F \leq 1.$
\end{proof}

\subsection{Supports in special sets and restricted tube domains}\label{sec:restrtube} 
The reader should be puzzled by now by the way too
abstract and useless level of this article. It is time perhaps for some examples of phase regularity applied to measures supported by three basic convex shapes in euclidean space: the orthant, the euclidean $l_2$-ball and the $l_1$-ball. The first one will especially be dear to control theorists, because it contains in the particular case of one dimension familiar computations of Laplace transforms. We include the euclidean ball as the commonly occurring domain for measures, and the $l_1$ ball as it results in trigonometric moment data for phase functions, which is an appealing set up for entropy optimization.

To remove some normalizations that encumber the computations we break with the generic convention $T_\Omega = \R^d + i \Omega$. Instead, at the beginning of each example, we specify the domain, redefine the Fantappi\`e transforms and then summarize the derivation which was detailed in the Section \ref{sec:phasereg} to obtain the final result.

\subsubsection{The orthant}\label{sec:orthant} Let $\Omega = (0,\infty)^d$ be the open positive orthant in $\R^d$, self-dual in the sense
$\Omega^\ast = {\overline{\Omega}}$, the closure of itself. Szeg\"o's kernel of the tube domain over $\Omega$ is 
$$ S(z,w) = \frac{1}{(2\pi)^d} \int_{\Omega^\ast} e^{i(z-\overline{w})\cdot u} du = \frac{1}{(2\pi i)^d} \prod_{k=1}^d \frac{1}{\overline{w_k}-z_k}, \ \ z,w \in T_\Omega = \R^d + i \Omega.$$

With the selection of the reference point $\alpha = (i,i,...,i) \in i\Omega$, Herglotz kernel becomes
\begin{align*} H(z, u; \alpha) &= 2 \frac{S(z,u)S(u,\alpha)}{S(z,\alpha)}  -  \frac{\abs{S(u,\alpha)}^2 }{S(\alpha, \alpha)} \\
&= \frac{1}{(2\pi)^d} \left[ 2 \prod_{k=1}^d \frac{1-iz_k}{(u_k-z_k)(u_k+i)} -\prod_{k=1}^d \frac{2}{1+u_k^2}\right] \\
&= 2 \prod_{k=1}^d \frac{1}{2\pi i} \left(\frac{1}{u_k-z_k} -\frac{1}{u_k+i}\right) - \prod_{k=1}^d \frac{1}{2\pi i}\left(\frac{1}{u_k-i}-\frac{1}{u_k+i}\right).
\end{align*}

In particular, for $d=1$ we recover the familiar Szeg\"o kernel $S(z,w) = \frac{1}{2\pi i} \frac{1}{\overline{w}-z}$ of the upper
half plane, and Herglotz kernel becomes
\[
 H(z,u;i) = \frac{1}{\pi i} \left(\frac{1}{u-z} - \frac{u}{u^2+1}\right).
\]
Let $\Phi(z)$ be an analytic function, mapping the open upper half-plane into itself and satisfying $\lim_{s \rightarrow \infty} \Phi(is) = 0.$ Then $\ln \Phi(z)$ is well defined, with $\Im \ln \Phi(z) \in (0,\pi)$. The argument preceding Proposition
\ref{exp-transform} remains valid, with the result
\[
 \Phi(z) = \exp [i F(z)], \ \ \Im z>0,
\]
where $F(z)$ is analytic, $\Re F(z) \in (0,\pi)$ and
\[ 
F(z) = i \Im F(i) + \int_\R H(z,u:i) \phi(u) du,
\]
and $\phi \in L^\infty(\R), \ 0 \leq \phi \leq \pi.$ In conclusion, we obtain the representation
\[
 \Phi(z) = e^{\Im F(i)} \exp \int_\R \left(\frac{1}{u-z} - \frac{u}{u^2+1}\right)\frac{\phi(u)}{\pi} du,\]
a formula already invoked in (\ref{partial-exp}).

\subsubsection{The $l_2$ ball}\label{sec:ball}  Let $\mu$ be a positive Borel measure supported by the closed unit ball $\mathbf b$ of ${\R}^d$, and
denote by
$$ a_\alpha (\mu) = \int_{\mathbf b} x^\alpha d\mu(x), \ \ \ \alpha \in {\N}^d.$$
We consider a version of the Fantappi\`e transform of $\mu$:
$$ {\fant}(\mu) (z) = \int_{\mathbf b}  \frac{d\mu(x)}{1- x \cdot z},$$
where $u \cdot v = u_1 v_1 + ...+ u_d v_d.$

Note that $ {\fant}(\mu)$ is an analytic functions defined in the open unit ball $\mathbf B$ of ${\mathbf C}^d$.
Its Taylor series expansion at $z=0$ is reducible, modulo universal constants, to the moments of $\mu$:
$$ {\fant}(\mu) (z) = \sum_\alpha \frac{\abs{\alpha}!}{\alpha !} a_\alpha (\mu) z^\alpha.$$

Remark also that, for all $z \in  {\mathbf B}$:
$$ \Re {\fant}(\mu) (z) = \int_{\mathbf b} \frac{1 - \Re x \cdot z}{\abs{1-x\cdot z}^2} d\mu (x) \geq 0.$$
By the maximum principle for pluri-harmonic functions, equality sign can happen only if ${\fant}(\mu) (z)$
is identically equal to a purely imaginary constant, which is impossible, since ${\fant}(\mu) (0) = \mu({\mathbf b})$.

Define the function 
\[F(z) \triangleq -i \ln i \Phi(\mu)(z)\] 
on the unit ball, such that $i \fant (\mu)(z) = \exp i F(z)$, with $\Re F(z) \in [0,\pi]$ for $z \in B$. The function $F(z)$ is analytic in the ball, and has a positive, bounded real part there. By the generalized Riesz-Herglotz formula, see \cite{koranyi1965,mccarthy2005}, we infer:
\[F(z) = i \Im F(0) + \int_{\partial {\mathbf B}} [2 S(z,w) -1] \phi(w) d\sigma(w),\]
where $\partial {\mathbf B}$ is the unit sphere, $\sigma(w)$ is the surface element on $\partial {\mathbf B}$, normalized to have mass
equal to one, 
$$ S(z,w) = \frac{1}{(1- z \cdot \overline{w})^d}$$
is Szeg\"o's kernel of the ball, and $\phi(w)$ is a measurable function on the sphere, satisfying
$$ 0 \leq \phi(w) \leq \pi, \quad w \in \partial {\mathbf B}.$$

Let us deal first with the free term, similarly as we did in the one-dimensional case:
$$ {\fant}(\mu) (0) = \mu({\mathbf b}),$$
and therefore
\[ F(0) = -i \ln i \mu(\mathbf{b}) = \underbrace{- i\ln \mu(\mathbf{b})}_{i \Im F(0)} + \frac{\pi}{2}. \]
The total mass of $\phi(w)$ is now easily computed by setting $z = 0$ in the Riesz-Herglotz formula, to obtain
\[
-i \ln i \mu(\mathbf{b}) + \frac{\pi}{2} = -i \ln i \mu(\mathbf{b}) + \int_{\partial \mathbf B} \phi(w) d\sigma(w),
\]
or 
\[
\int_{\partial \mathbf B} \phi(w) d\sigma(w) = \frac{\pi}{2}.
\]

The Fantappi\'e transform of the measure $\mu$ simplified by substituting the exact values for the free term and the total mass of the phase function $\phi$:
\begin{align*}
  i\fant (\mu)(z) &= \exp i F(z) \\
  &= \exp\left[ \ln \mu(\mathbf{b})  + i\int_{\partial {\mathbf B}} [2 S(z,w) -1] \phi(w) d\sigma(w)\right] \\
  \fant (\mu)(z)&= \mu( \mathbf{b} ) \exp \left\{ i \int_{\partial {\mathbf B}} [2S(z,w) - 1] \phi(w) d\sigma(w)  - i\frac{\pi}{2}\right\}\\
   &= \mu(\mathbf{b}) \exp \left\{2i \int_{\partial {\mathbf B}} [S(z,w) - 1]\phi(w) d\sigma(w)\right\}
\end{align*}
This formula relates, via a triangular, non-linear transformation, the moments $(a_\mu (\alpha))$ of
$\mu$ to the moments $(a_\phi (\alpha))$ of the density $\phi$. Additionally, note that the Koranyi-Puk\'anszky theorem asserts that all the multivariate moments of $\phi$ at mixed-sign indices are zero, i.e., for a multi-index $\alpha$, $a_\phi(\alpha) = 0$, except possibly when $\forall i, \alpha_i \geq 0$, or $\forall i, \alpha_i \leq 0$.

\subsubsection{The $l_1$ ball}\label{sec:polydisk} Let $\Delta = \{ x \in {\R}^d; \ \abs{x_1} + ...+\abs{x_d} \leq 1 \}$ and consider a positive measure
$\mu$ supported by $\Delta$. Its Fantappi\`e transform is analytic in the unit polydisk ${\disk}^d$, and has there
a power series expansion
$$ {\fant}(\mu) (z) = \sum_\alpha \frac{\abs{\alpha}!}{\alpha !} a_\mu (\alpha) z^\alpha.$$
Arguing as in the case of the ball, the function $\ln {\fant}(\mu)$ is analytic in the open polydisk, and has
non-negative imaginary part there bounded above by $\pi$. The analog of Riesz-Herglotz formula (as derived for the first time
by Koranyi and Puk\'anszky \cite{koranyi1963}) yields a measurable function $\phi$ defined on the unit torus ${\T}^d$, with values
in the interval $[0,\pi]$, so that:
$$ {\fant}(\mu) (z) = \mu(\Delta) \exp \left\{ 2i \int_{{\T}^d}  [ \Pi(z,w) -1] \phi(w) d\theta(w)\right\},$$
where this time
$$ \Pi(z,w) = \prod_{j=1}^d \frac{1}{1- z_j \overline{w}_j},$$
and
$$ d\theta(w) = \prod_{j=1}^d \frac{d w_j}{2 \pi i w_j}.$$

Let $a_\phi(\alpha) = \int_{{\T}^d} \phi(w) \overline{w}^\alpha d\theta(w)$ denote the 
Fourier coefficients of the function $\phi$ (i.e. its trigonometric moments on the torus).

At the level of generating series we obtain the following transform
$$ \sum_\alpha \frac{\abs{\alpha}!}{\alpha !} a_\mu (\alpha) z^\alpha = \mu(\Delta) \exp [ 2i \sum_{\alpha \neq 0} a_\phi(\alpha) z^\alpha]. $$ Consider a simple example.

\begin{example}
  To verify that the above formulas are correct, we consider the 1D
  case, with the Dirac measure $d\mu = c \delta_a$, where $c>0, a \in
  [-1,1]$. The Fantappi\`e transform is the analytic function in the
  unit disk:
$$ f(z) = \frac{c}{1-az}.$$ Since $f(z)$ has positive real part, $if(z)$ has positive imaginary part in the disk.
Whence $\ln if(z)$ is well defined, analytic, and has imaginary part
in the interval $[0,\pi]$, thus the classical Riesz-Herglotz formula
is applicable to the function $-i\ln i f(z)$:
$$ -i \ln i f(z) = i \Im [-i \ln i f(0)] + \int_{\T} \frac{1+ z \overline{w}}{1-z\overline{w}} \phi(w) \frac{ dw}{ 2 \pi i w}.$$
Above, $\phi(w) = \Re (-i \ln i f(w)) $ is a measurable function on
the torus with values in $[0,\pi]$. Note that $f(0) = c$, so that $
\ln i f(0) = \ln c + i \pi/2$, and $i \Im [-i \ln i f(0)] = -i \ln
c.$ By evaluating $z=0$ in the formula we obtain
$$ -i \ln c + \pi/2 = -i \ln c + \int_{\T}  \phi(w) \frac{ dw}{ 2 \pi i w},$$ that is 
$$ \pi/2 = \int_{\T}  \phi(w) \frac{ dw}{ 2 \pi i w}.$$ Finally we get
\begin{align*}
  \ln i + \ln f(z)
  &= \ln c + i \int_{\T} \frac{1+ z \overline{w}}{1-z\overline{w}} \phi(w) \frac{ dw}{ 2 \pi i w}\\
  \ln f(z)  &= \ln c + i \int_{\T} \left[\frac{1+ z \overline{w}}{1-z\overline{w}}-1\right] \phi(w) \frac{ dw}{ 2 \pi i w} \\
  &= \ln c + 2i \int_{\T} \frac{z \overline{w}}{1-z\overline{w}}
  \phi(w) \frac{ dw}{ 2 \pi i w} \\
  &= \ln c + 2i \int_{\T}
  \left[\frac{1}{1-z\overline{w}}-1\right] \phi(w) \frac{ dw}{ 2 \pi i w},
\end{align*}
which
is consistent with our general formulas.

One step further, we can easily compute the positive ($n>0$) Fourier
coefficients of $\phi(w) = \Re (-i \ln i f(w))$:
$$ \int_{\T} \phi(w) \overline{w}^n  \frac{ dw}{ 2 \pi i w} = \int_{\T} \frac{-i}{2} \ln i f(w) \overline{w}^n  \frac{ dw}{ 2 \pi i w} = $$ $$ 
\frac{1}{2i} \int_{\T} \ln \frac{c}{1-aw} \overline{w}^n \frac{ dw}{
  2 \pi i w} = \frac{1}{2i} \frac{ a^n}{n}.$$

The final verification:
$$ \frac{c}{1-az} = c \exp \sum_{n=1}^\infty \frac{a^n}{n} z^n.$$
\end{example}

\subsection{Partial Fantappi\`e transform}\label{sec:partial-fantappie} 
In this section we continue analytically the Fantappi\`e transform of a measure
supported by a convex cone in a single direction, treating the rest of the variables as parameters, with a double benefit:
a simple and well known formulas in 1D, and a tight control of the growth of the phase function. We closely follow below the
article \cite{arons1956}, although similar computations have appeared much earlier in the work of Nevanlinna and Verblunsky.

\subsubsection{Regularization by parametrized single-variable transforms}

The setting is the same: $\Gamma \subset \R^d$ is a closed, solid, acute convex cone and $\mu$ is a finite positive measure supported by
its polar cone $\Gamma^\ast$. We consider the analytic extension of the Fantappi\`e transform:
\[
\Phi(-z,y) = \int_{\Gamma^\ast} \frac{d\mu(x)}{-z+ x \cdot y}, \ \  y \in \Gamma, \ \Im z >0.
\]
Since $$ \frac{1}{-\overline{z}+ x \cdot y} -\frac{1}{-z+ x \cdot y} = \frac{\overline{z} - z}{\abs{z-x\cdot y}^2}$$
$$ (y \in \Gamma, \ \Im z >0)\ \ \Rightarrow \ \   \Im \Phi(-z,y) >0.$$

Thus, for any fixed $y \in \Gamma$, the function $z \mapsto i\Phi (-z,y)$ preserves the upper half-plane
and hence it can be represented for $\Im w >0$ as
\begin{equation}\label{partial-exp}
\int_{\Gamma^\ast} \frac{d\mu(x)}{-z+x\cdot y} =  \Phi(-z,y) = C(y) \exp \int_\R \left(\frac{1}{t-z} - \frac{t}{1+t^2}\right) \phi_y (t) dt, 
\end{equation}
where $C(y) >0$ and $0 \leq \phi_y(t) \leq 1$ both depending measurably on $y$, respectively $y$ and $t$,
see \cite{arons1956}. Both functions $C(y), \phi_y(t)$ are uniquely determined by $\Phi(z,y)$, hence
by the measure $\mu$. For illustration, we provide a simple example of a point mass:

\begin{example}
  Take $\mu = c \delta_0$, where
  $c>0$. First we obtain by direct integration
$$ \frac{-1}{z} = \exp  \int_0^\infty (\frac{1}{t-z} - \frac{t}{1+t^2}) dt,$$
hence
$$ \int_{\Gamma^\ast} \frac{c  \delta_0 (x)}{-z+ix\cdot y} = c \exp  \int_0^\infty (\frac{1}{t-z} - \frac{t}{1+t^2}) dt,$$
obtaining
$$ C(y) = c, \ \phi_y(t) = \chi_{[0,\infty)}(t),  \ \  y \in \Gamma.$$
It is rather annoying that for such a simple measure as a point mass,
the phase function has an unbounded support.  Fortunately, there is a
simple remedy, derived from an observation of Verblunsky
\cite{verblunsky1947,delange1950}.
\end{example}

\begin{thm} Let $\mu$ be a finite positive measure supported on the cone  $\Gamma^\ast$. For every $y \in \Gamma$
there exists a phase function $\xi_y \in L^1([0,\infty), dt), \ 0 \leq \xi_y \leq 1,$ measurably depending on $y$, 
such that
\begin{equation}\label{verblunsky}
1 + \int_\Gamma \frac{d\mu(x)}{x \cdot y - z} = \exp \int_0^\infty \frac{\xi_y(t) dt}{t-z},\ \ \Im z>0.
\end{equation}
Moreover, if $\int_{\Gamma^\ast} \abs{x}^n d\mu(x)<\infty$ for some $n \in \N$, then $\int_0^\infty t^n \xi_y(t) dt <\infty$
for all $y \in \Gamma$.
\end{thm}

\begin{proof} Fix a point $y \in \Gamma$ and denote by $\mu_y$ the push forward of the measure
$\mu$ via the map $x \mapsto x \cdot y.$ Specifically, for a test function $f \in C_0([0,\infty))$:
\begin{align}
  \int_{\Gamma^\ast} f(x\cdot y) d\mu(x) = \int_0^\infty f(t)
  d\mu_y(t).\label{eq:pushforward}
\end{align}
In these terms, Fantappi\`e's transform
of the measure $\mu$ becomes 
$$\Phi(z,y) = \int_0^\infty \frac{d\mu_y(t)}{z+t}.$$ According to Verblunsky's theorem \cite{arons1956}, there exists a
function  $\xi_y \in L^1([0,\infty), dt), \ 0 \leq \xi_y \leq 1$ with the property
$$1 + \Phi(-z,y) = \exp \int_0^\infty \frac{\xi_y(t) dt}{t-z},\ \ \Im z>0.$$ The boundary limit operations
producing $\xi_y$ from $\mu_y$ imply that the dependence $y \mapsto \xi_y$ is (weakly)
measurable, as a map from $\Gamma$ to $L^1([0,\infty), dt)$.
Finally, Theorem A.b) of \cite{arons1956} implies the finiteness of the first $n+1$ moments of every $\xi_y$,
provided that the moments of $\mu_y$ of the same order are finite.
\end{proof}

Returning to our simple example, $\mu = c\delta_0$, we find this time
$$ 1+\Phi(-z,y) = 1 + \int_{\Gamma^\ast} \frac{c \delta_0(x)}{x \cdot y - z} = 1 - \frac{c}{z} = \exp \int_0^c \frac{dt}{t-z},$$
whence
$$ \xi_y = \chi_{[0,c]}, \ \ y \in \Gamma.$$

Assume that $\int_{\Gamma^\ast} \abs{x}^n d\mu(x)<\infty$ for a positive value of $n$ and denote the initial moments of $\mu$ as:
$$ \gamma_\alpha = \int_{\Gamma^\ast} x^\alpha d\mu(x), \ \ \abs{\alpha} \leq n.$$ Similarly, denote
$$ c(y)_j = \int_0^\infty t^j \xi_y(t) dt, \ \ 0 \leq j \leq n.$$

Then one can identify asymptotically the 
series expansion (in a wedge with vertex at $z=0$) of the two terms in (\ref{verblunsky}), obtaining the algebraic relation
\begin{equation}\label{Markov-recurrence}
 1 - \sum_{k=0}^n \frac{k!}{z^{k+1}}( \sum_{\abs{\alpha}=k} \frac{y^\alpha}{\alpha!} \gamma_\alpha) + O(z^{-n-2}) = 
\exp ( - \sum_{j=0}^n \frac{c_j(y)}{z^{j+1}}).
\end{equation}
 For details see again \cite{arons1956}.
In particular, by equating the coefficients of $z^{-1}$ one finds Verblunsky's identity
$$ c_0(y) = \gamma_0, \ \ y \in \Gamma.$$

Note that, after expanding the exponential series, the moment $c_j(y)$ is given by a universal 
polynomial function in the variables $\gamma_\alpha, \  \abs{\alpha} \leq j.$ It is exactly this system of 
polynomial equation which was discovered and exploited by Markov (in dimension one).

\subsubsection{Inversion through Radon transform}

Using entropy optimization, for each $p \in \Gamma$, we obtain functions $\xi^\ast_p$ that approximate Aronszajn-Donoghue phase functions $\xi_p$  of push forward measures $\mu_p$ defined by \eqref{eq:pushforward}. The inversion procedure given in Section \ref{sec:oned} would be able to recover $\mu^*_p$ measures, however, piecing together approximation $\mu^\ast$ from ``slices'' $\mu^\ast_p$ would be a challenging task. Instead, we seek to recover the Radon transform $\radon \mu^\ast (p, t)$ of approximation $\mu$, which is then inverted by one of the standard algorithms for inverse Radon transform.

We start by relating the partial Fantappi\`e transform to Cauchy transform
\begin{align*}
  \fant \mu(p, -z) = \int_{\Gamma^\ast} \frac{d\mu(x)}{-z + p\cdot x} = \cauchy \mu_p(z) 
  = \exp \cauchy \xi_p(z) - 1,
\end{align*}
with a slight abuse of notation for the Cauchy transform, see \eqref{verblunsky}. To avoid the singular integral in passing from complex to the real Fantappi\`e directly, we sum the Plemelj-Sokhotski formulas \eqref{eq:ps-realline} to obtain
\begin{align*}
  \fant \mu(p, p_0) = \cauchy \mu_p(-p_0) &= \lim_{\epsilon \downarrow 0} \frac{1}{2}\left[ \cauchy \mu_p(-p_0 - i\epsilon) + \cauchy \mu_p(-p_0 + i\epsilon)\right] \\
  &= \exp\left[ -\pi \hilb \xi_p(-p_0)\right] \cos \left[ \pi \xi_p(-p_0) \right] - 1,
\end{align*}
using the derivation analogous to derivation of \eqref{eq:onedinverse}, where we defined function $f_p$ just to relieve notation for the rest of the procedure.

To connect the Fantappi\`e transform to the Radon transform of a measure, we again follow Henkin and Shananin \cite{Henkin1990}. For a positive measure $\mu$ with density $\xi$ supported in the positive orthant $\R_+^n$, define the Radon transform as
\begin{align}
  \label{eq:radon}
  \radon \mu (\theta, s) &\triangleq \int_{H(\theta, s)} \xi(x)d\sigma(x),  \\
&= \int_{\R^n} \delta(s - x \cdot \theta) \xi(x) d\sigma(x),
\end{align}
where $H(\theta,s) \triangleq \{ x \in \R^n : x \cdot \theta = s \}$ is the integration hyperplane and $d\sigma(x)$ its surface area measure.
Henkin and Shananin give the following relation between Fantappi\`e and Radon transforms of rapidly decaying measure $\mu$: 
\begin{align}
  \fant \mu(p,p_0) = \int_0^\infty \frac{\radon\mu(p,\tau)}{\tau +
    p_0} d\tau,\label{eq:hs-radon-fantappie}
\end{align}
valid for $p_0 > 0$, $p \in \R^n_+$. This expression can be interpreted as a composition of Hilbert and Radon transforms
\[
\fant \mu(p,p_0) = [ \hilb \radon \mu(p,\cdot) ](-p_0),
\]
where Hilbert transform is taken along the offset parameter. As $-\hilb^2$ is the identity operator, by applying another Hilbert transform, we can evaluate the Radon transform as
\[
\radon \mu(p,p_0) = -\frac{1}{\pi} [\hilb f_p](p_0),
\]
where we define
\[
f_p(p_0) \triangleq \exp\left[ -\pi \hilb \xi_p(p_0)\right] \cos \left[ \pi \xi_p(p_0) \right] - 1.
\]
As mentioned before, Hilbert transform can be efficiently evaluated using FFT. Therefore, for each selected $p$, we can evaluate the Radon transform along $p_0$ axis. One might validly ask why we decided to go through this labyrinthine process only to evaluate a Radon transform of $\mu$. The answer lies in the method used to obtain moment data. If the original measure $\mu$ has a density fully accessible for arbitrary Radon-type measurements, as it is in medical tomography, then the entire procedure is superfluous as we can access $\radon_\mu(p)$ directly at any $p \in \Gamma$. However, for singular measures and in some settings, e.g., invariant measures on attractors of dynamical systems, Radon-measurements are not directly possible and the described procedure becomes an acceptable path to reconstruction.

An unfortunate obstacle prevents us for completing the process by invoking a readily-available inversion algorithm for the Radon transform. The relation between the Fantappi\`e and Radon transform \eqref{eq:hs-radon-fantappie} holds only in the positive orthant $(p_0,p) \in \R^{n+1}_+$, which is, in general, not enough for a numerically-stable reconstruction \cite{natte2001}. It is possible that this obstacle can be removed by considering certain symmetries of the problem. However, these considerations would lead us too far from the central theme of this paper and we plan to explore them in a subsequent paper.

\section{Conclusions}
\label{sec:conclusions}

In this paper, we presented an approach aimed at representing singular measures using bounded densities. Our goal was to use existing entropy optimization methods to solve truncated moment problems for singular measures. Previously, the entropy optimization could not be applied to singular measures due to lack of convergence in the optimization procedure

The paper adds two steps that bookend the entropy optimization. The regularization step uses a triangular, recursive transformation on the input moment set to produce the conditioned moment set which is a feasible input for entropy optimization. The inversion step uses the density that solves the entropy optimization constrained by conditioned moments, and recovers an approximation to the original measure.

We presented both steps in detail for the support in a one-dimensional space. Two different settings, unbounded and bounded supports, were analyzed, resulting in, respectively, a more general formulation using power moments, and a numerically favorable formulation using trigonometric moments.

We generalize the regularization step to multivariate domains, in particular supports of measures in wedges and particular compact domains in $\R^d$. The inversion, however, proved to be a technically more demanding task, with the direct generalization of the Plemelj-Sokhotski formulas requiring a detailed application of theory of Clifford algebras. We did not tackle such generalization in this paper, nevertheless, we presented an outline of a simpler tomographic process. It would reduce a multivariate inversion problem to a family of inversion along rays in the domain, which can be solved using the one-dimensional method. The information based on ray transforms could then be integrated into the approximation of the original measure using methods of tomography.

This paper is the first part of this research effort. We plan to explore the full solution to the multivariate inversion step using tomography in one of the follow-up papers. Furthermore, a future paper should formulate error analysis and present numerical confirmation of the usefulness of our method, applied to concrete examples.

\appendix

\section{Miller-Nakos algorithm for exponentiation of a power series}
\label{sec:miller}

Computing coefficients of an exponentiated (formal) power series can be performed recursively, i.e., from knowledge of the coefficients of the base and coefficients of the lower coefficients. The original algorithm in a single variable is given by Peter Henrici who attributes it to J.C.P.\ Miller \cite{henrici1956}. The extension to the multivariate case is due to George Nakos \cite{nakos1993}, which is unfortunately published in a journal that is not easily accessible. Therefore, we give the algorithm in its entirety here, stressing that all the original ideas were present in the above papers. Ours are the choice of notation and the (trivial) extension to the case of series with zero free terms, mentioned at the end of the section.

Let $\alpha, \beta, \mu, \gamma \in \N^d_0$ denote multi-indices and let the basis multi-index be $\epsilon_i = (0,0,\dots,0,1,0,\dots,0)$ where $1$ is at the $i$-th position. Unless noted otherwise, sums over multi-indices range over all multi-indices $\N_0^d$. Furthermore, denote $\partial_i = \frac{\partial}{\partial x_i}$. 
\begin{thm}[J.C.P.\ Miller, G.\ Nakos]
Let $A(x)$ and $B(x)$ be multivariate power series with non-zero free terms, given by expressions 
\begin{align*}
  A(x) \triangleq \sum_{\alpha} a_\alpha x^\alpha &&  B(x) \triangleq \sum_{\beta} b_\beta x^\beta,
\end{align*}
with $\alpha_0 \not = 0$, $\beta_0 \not = 0$.

If the power series are related by equation
\[
B(x) = A(x)^k,
\]
then coefficients $b_\beta$ can be computed recursively by expressions
\begin{equation}
\begin{aligned}
  b_0 &= a_0^k \\
  b_\mu &= \sum_{0 < \gamma \leq \mu} \frac{1}{a_0}\left[ (k+1) \frac{\abs{\gamma/\mu}}{\abs{\mu/\mu}} - 1\right] a_\gamma b_{\mu - \gamma},
\end{aligned}\label{eq:nakos}
\end{equation}
where \[\abs{\alpha / \beta} = \sum_{ i : \beta_i \not = 0} \alpha_i / \beta_i.\footnote{Particularly, $\abs{\alpha / \alpha}$ counts the nonzero elements in $\alpha$.}\]
\end{thm}
\begin{proof}

The proof of the algorithm rests on the identity
\begin{align}
  \label{eq:millerid}
  A(x) \partial_i B(x) = k B(x) \partial_i A(x),
\end{align}
derived by expanding $\partial_i B(x)$ by chain rule and multiplying both sides by $A(x)$. Series expansions for derivatives $\partial_i A(x)$, and analogously for $\partial_i B(x)$, is
\[
 \partial_i A(x) = \sum_{\alpha} a_\alpha \alpha_i x^{\alpha - \epsilon_i}.
\]
  Expanding \eqref{eq:millerid} into series we obtain 
\[
\left(\sum_\alpha a_\alpha x^\alpha\right)  \left(\sum_\beta b_\beta \beta_i x^{\beta - \epsilon_i} \right)= 
k  \left(\sum_\beta b_\beta x^\beta\right)  \left(\sum_\alpha a_\alpha \alpha_i x^{\alpha - \epsilon_i}\right).
\]

Introduce following change of indices: on the left hand side over $\beta$, $\beta - \epsilon_i \mapsto \beta$, $\beta \mapsto \beta + \epsilon_i$, $\beta_i \mapsto \beta_i+1$, and the analogous substitutions on the right hand side in sum over $\alpha$, to obtain the identity
\[
\sum_\alpha a_\alpha x^\alpha \sum_\beta (\beta_i + 1) b_{\beta+\epsilon_i} x^\beta = 
k \sum_\beta b_\beta x^\beta \sum_\alpha (\alpha_i+1) a_{\alpha + \epsilon_i} x^\alpha.
\]

At this point, the goal is to compute coefficient $b_\mu$ from knowledge of $a_\alpha$ and $b_\gamma$ for $\gamma < \mu$. The products of power series $\sum_\alpha a_\alpha x^\alpha \sum_\beta b_\beta x^\beta = \sum_\gamma c_\gamma x^\gamma$ are computed using convolution over coefficients $c_\gamma \triangleq \sum_{\alpha + \beta = \gamma} a_\alpha b_\beta$.
 The coefficients at index $\mu - \epsilon_i$ on both sides of the above identity have to match, yielding coefficient identity
\[
\sum_{\alpha + \beta = \mu - \epsilon_i} a_\alpha b_{\beta + \epsilon_i} (\beta_i + 1) = k \sum_{\alpha + \beta = \mu-\epsilon_i} b_\beta a_{\alpha+\epsilon_i}(\alpha_i + 1).
\]
The expressions can be simplified by another re-indexing, taking into account constraints on indices in summation, and symmetry in summation: on LHS, $\alpha \mapsto \gamma$, $\epsilon_i + \beta \mapsto \mu - \gamma$, on RHS $\beta \mapsto \mu - \gamma$, $\epsilon_i + \alpha \mapsto \gamma$, to obtain
\[
\sum_{0 \leq \gamma \leq mu} a_\gamma b_{\mu - \gamma} (\mu_i - \gamma_i) = k \sum_{0 \leq \gamma \leq \mu} b_{\mu - \gamma} a_\gamma \gamma_i,
\]
or
\[
\sum_{0 \leq \gamma \leq \mu} a_\gamma b_{\mu-\gamma} \left[ \mu_i - (k+1) \gamma_i \right] = 0.
\]
To solve for $b_\mu$, extract $\gamma = 0$ case from the sum to obtain
\begin{align*}
  a_0 b_\mu \mu_i &= \sum_{0 < \gamma \leq \mu} a_\gamma b_{\mu - \gamma}\left[ (k+1) \gamma_i - \mu_i \right] \\
 b_\mu &= \frac{1}{a_0} \sum_{0 < \gamma \leq \mu} a_\gamma b_{\mu - \gamma} \left[ (k+1) \frac{\gamma_i}{\mu_i} - 1\right].
\end{align*}
Since this expression is valid if and only if $\mu_i \not = 0$, we can sum over all such cases to obtain
\[
\abs{\mu / \mu} b_\mu = \frac{1}{a_0} \sum_{0 < \gamma \leq \mu} a_\gamma b_{\mu - \gamma} \left[ (k+1) \abs{\gamma/\mu} - \abs{\mu/\mu} \right],
\]
which, through dividing by $\abs{\mu / \mu}$, yields the expression \eqref{eq:nakos}.
\end{proof}

The moment expansion series that we use have a zero free term, which is essential for solving for $b_\mu$ in the last step of the proof. This can be resolved by adding an extra step which uses the binomial expansion
\[
[(A(x) + 1) - 1]^n = \sum_{0 \leq k \leq n} {n \choose k} (-1)^{n-k} \left[A(x) + 1\right]^k,
\]
where $A(x)$ is the series without a free term. In general, this step does involve $n$ extra computations, however, the moment conversions such as \eqref{eq:1dmomentconversion} require all the powers between $0$ and $n$ anyway, so there is no additional cost involved for our purposes.

\section{Fast Iterative Algorithm for Entropy Optimization}
\label{sec:fime}

This algorithm has been described in \cite{bandy2005}, and it is based on explicit discretization of the entropy functional, which reduces computation of moments to a matrix multiplication. For completeness, we present it here in a distilled form.

First, assume the domain is $[0,1]$ interval, and choose points $x_k$ with quadrature weights $w_k$ for $k = 1,\dots, K$, i.e.,
\begin{align*}
  \int_0^1 f(x) dx \approx \sum_{j=1}^J f(x_j) w_j.
\end{align*}
For an arbitrary distribution $p(x)$, we will write $\mathbf{p} = (p_j)$, $p_j := p(x_j)$, and $\mathbf{\tilde p} = (\tilde p_j)$, $\tilde p_j = w_j p_j$. To evaluate moments of $p(x)$ we employ the matrix $\mathbf A  = (a_{ij})$, whose rows are evaluations of monomials on the array $x_k$. For first $N$ power moments, $\mathbf A$ will be a row-truncated Vandermonde matrix $a_{ij} = x_j^i$, where $i = 1, \dots, N$, and $j = 1, \dots, K$. The vector of moments $\mathbf \mu = (\mu_i)$ for a discretized distribution $p_j$ is easily evaluated by taking the product $\mathbf \mu = \mathbf A \cdot \mathbf{ \tilde p}$.

Let $\alpha_i$, $i=1,\dots, N$ be the set of the Lagrange multipliers (dual variables), in which the entropy optimization is unconstrained. The primal is then evaluated by function
\begin{align}\label{eq:primal}
  p(\mathbf \alpha)_j = \exp \left[ (\mathbf A^T \cdot  \mathbf \alpha)_j - 1 \right],
\end{align}
with the goal of finding $\alpha^*$ such that $p_j(\alpha^*) \approx p^*_j$.
Constraint deviation vector is 
\begin{align}\label{eq:constraint}
  h_i(\alpha) = \left[ \mathbf A \cdot \mathbf{p}(\alpha) \right]_i - \mu_i.
\end{align}
The optimization program strives to achieve 
\begin{align*}
  \min \left(\mathbf 1^T \mathbf p(\alpha) - \mu^T \alpha\right),
\end{align*}
by cyclically updating components of $\alpha$. Denote $k$th iteration of $\alpha$ by $\alpha^{(k)}$.

Let $k$ be a step counter, and set $i = (k \mod N) + 1$ to be the index of the Lagrange multiplier updated in the $k$th step. Fix the convergence tolerance $\epsilon > 0$, and denote initial step by $k=0$. The initial vector $\alpha^{(0)}$ can be chosen as random numbers in some interval, a constant vector, or some other vector.

Compute the correction factor
\begin{align}\label{eq:correctionfactor}
  \lambda^{(k)} = \ln \frac{ \mu_i}{ \left[\mathbf A \cdot \mathbf p(\alpha^{(k)})\right]_i },
\end{align}
and update $i$th Lagrange multiplier
\begin{align*}
\alpha^{(k+1)}_i = \alpha^{(k)}_i + \lambda^{(k)}& \\
\alpha^{(k+1)}_j = \alpha^{(k)}_j &,\ \text{for } j \not = i.     
\end{align*}
If $\norm{ h(\alpha^{(k+1)}) } < \epsilon$, then $\alpha^* =
\alpha^{(k+1)}$, otherwise, increase $k$ by one, and restart from
computation of the correction factor.

The paper \cite{bandy2005} asserts that the algorithm converges when $\mu_i > 0$ and $a_{ij} \in [0,1]$, $\forall i$, or when $\mu_i < 0$ and $a_{ij} \in [-1,0]$, $\forall i$. When this is not the case, the authors provide a pre-conditioning step that modifies $\mathbf A$ and $\mu$ to ensure convergence.

The presented algorithm can be extended to cases where generalized moments are taken, instead of power moments. In those cases, the moments are not necessarily positive, nor is the matrix $\mathbf A$, so they have to be rescaled before the correction factor formula \eqref{eq:correctionfactor} can be used.

To generalized moments on another interval, let $T_i(x)$ be linearly independent functions which are used to generate generalized moments. The matrix $A$ is then given by elements $a_{ij} = T_i(x_j)$ for $i = 1,\dots,N$, $j = 1,\dots,K$. Choose a positive constant $\delta > 0$ and
let
\begin{align*}
  \text{(needed offset)} && u_i &= -\min_{j}(a_{ij}) + \delta \\
  \text{(scale)} && M_i &= \max_{j} (u_j + a_{ij}) \\
  \text{(scaling factor)} && t_i &= \frac{1}{(M_i + \delta)} 
\end{align*}
The original paper used $\delta = 1$ and $t_i = [N(M_i + \delta)]^{-1}$ but we found that such settings result in somewhat slower convergence rates.

The conditioned matrix $\mathbf A^\prime$ and moment vector $\mu^\prime$ are then computed as
\begin{align*}
  a_{ij}^\prime &= t_i(u_i + a_{ij}) \\
  \mu_{i}^\prime &= t_i(u_i + \mu_{i}),
\end{align*}
which ensures convergence conditions.

Now the original program is modified by replacing the primal formula \eqref{eq:primal}, the correction factor formula \eqref{eq:correctionfactor}, and constraint equation \eqref{eq:constraint} by, respectively,
\begin{align*}
\mathbf p^\prime(\alpha) &= \exp \left[ (\mathbf A^{\prime T} \cdot  \mathbf \alpha)_j - 1 \right],\\
  \lambda^{\prime (k)} &= \ln \frac{ \mu^\prime_i}{ \left[\mathbf A^\prime \cdot \mathbf p^\prime(\alpha^{(k)})\right]_i },\shortintertext{and}\\
  h_i^\prime(\alpha) &= \left[ \mathbf A \cdot \mathbf{p}^\prime(\alpha) \right]_i - \mu_i.
\end{align*}
Note that in the constraint deviation, we evaluate the unconditioned moments of the conditioned primal. The justification is the fact that $\mathbf p^\ast = \mathbf p^\prime(\alpha)$ simultaneously solves the conditioned problem $\mathbf A^\prime \mathbf p = \mu^\prime$ and unconditioned problem $\mathbf A \mathbf p = \mu$, due to linearity of the pre-conditioning step. We, therefore, find the unconditioned constraint deviation condition more intuitive to use, which makes it easier to choose a desired convergence tolerance $\epsilon$.

\pdfbookmark{References}{References} 
\bibliography{references}
\bibliographystyle{elsarticle-num}

\end{document}